\documentclass[a4paper,reqno, 11pt]{amsart}
\usepackage[left=2.8cm,right=2.8cm,top=2.86cm,bottom=2.86cm]{geometry}
\usepackage{amsmath,amssymb,latexsym,esint,cite,verbatim,wasysym,mathrsfs}
\usepackage{microtype}
\usepackage{color,enumitem,graphicx}
\usepackage[colorlinks=true,urlcolor=blue, citecolor=red,linkcolor=blue,
linktocpage,pdfpagelabels,bookmarksnumbered,bookmarksopen]{hyperref}
\usepackage[hyperpageref]{backref}
\usepackage[english]{babel}
\renewcommand{\epsilon}{\varepsilon}

\newcommand{\pnorm}[2][]{\if #1'' \left|#2\right|_p \else \left|#2\right|_{#1} \fi}
\hypersetup{linkcolor=blue, colorlinks=true ,citecolor = red}

\DeclareMathOperator*{\loc}{loc}
\newcommand*\diff{\mathop{}\!\mathrm{d}}

\begin{document}

\title[On equations involving
the fractional $p(\cdot)$-Laplacian] {A-priori bounds and
multiplicity of solutions for nonlinear elliptic problems involving
the fractional $p(\cdot)$-Laplacian}

\author[K. Ho]{Ky Ho}\address{Ky Ho\newline Institute of Fundamental and Applied Sciences, Duy Tan University, Ho Chi Minh City 700000, Vietnam}
\email{hnky81@gmail.com}

\author[Y.-H. Kim]{Yun-Ho Kim*}
\address{Yun-Ho Kim\\
Department of Mathematics Education\\
Sangmyung University\\
Seoul 110-743, Republic of Korea}
\email{kyh1213@smu.ac.kr}

\thanks{* Corresponding author.}

\date{}
\subjclass[2010]{35B45, 35D30, 35J20, 35J60, 35J92, 46E35.}

\keywords{The fractional $p$-Laplacian; the $p(\cdot)$-Laplacian; fractional Sobolev
spaces with variable exponent; a-priori bounds; De Giorgi iteration; variational methods.}

\begin{abstract}
We obtain fundamental imbeddings for the fractional Sobolev space with variable exponent that is a generalization of well-known fractional Sobolev spaces. As an application, we obtain a-priori bounds and multiplicity of solutions to some nonlinear elliptic problems involving 
the fractional $p(\cdot)$-Laplacian.

\end{abstract}

\maketitle \numberwithin{equation}{section}
\newtheorem{theorem}{Theorem}[section]
\newtheorem{lemma}[theorem]{Lemma}
\newtheorem{definition}[theorem]{Definition}
\newtheorem{claim}[theorem]{Claim}
\newtheorem{proposition}[theorem]{Proposition}
\newtheorem{remark}[theorem]{Remark}
\newtheorem{corollary}[theorem]{Corollary}
\newtheorem{example}[theorem]{Example}
\allowdisplaybreaks

\newcommand{\abs}[1]{\left\lvert#1\right\rvert}
\newcommand{\norm}[1]{|\!|#1|\!|}
\newcommand{\Norm}[1]{\mathinner{\Big|\!\Big|#1\Big|\!\Big|}}
\newcommand{\curly}[1]{\left\{#1\right\}}
\newcommand{\Curly}[1]{\mathinner{\mathopen\{#1\mathclose\}}}
\newcommand{\round}[1]{\left(#1\right)}
\newcommand{\bracket}[1]{\left[#1\right]}
\newcommand{\scal}[1]{\left\langle#1\right\rangle}
\newcommand{\Div}{\text{\upshape div}}
\newcommand{\dotsr}{\dotsm}
\newcommand{\R}{{\mathbb R}}
\newcommand{\la}{{\lambda}}
\newcommand{\Zero}{{\mathbf0}}
\newcommand{\ra}{\rightarrow}
\newcommand{\ran}{\rangle}
\newcommand{\lan}{\langle}
\newcommand{\ol}{\overline}
\newcommand{\N}{{\mathbb N}}
\newcommand{\e}{{\varepsilon}}
\newcommand{\al}{{\alpha}}


\newcommand{\Assg}[1]{\textup{(g)}}
\newcommand{\AssF}[1]{\textup{(F#1)}}
\newcommand{\AssHJ}[1]{\textup{(HA#1)}}
\newcommand{\AssJ}[1]{\textup{(J#1)}}
\newcommand{\AssG}[1]{\textup{(Q#1)}}
\newcommand{\Assf}[1]{\textup{($f$#1)}}
\newcommand{\AssHA}[1]{\textup{(HJ#1)}}
\newcommand{\AssA}[1]{\textup{(A#1)}}
\newcommand{\Assa}{\textup{(a)}}
\newcommand{\Assw}[1]{\textup{($w#1$)}}
\newcommand{\AssV}{\textup{(V)}}
\newcommand{\AssAR}{\textup{(AR)}}
\newcommand{\AssJe}{\textup{(Je)}}
\newcommand{\AssH}[1]{\textup{(H#1)}}

\section{Introduction}
In the last two decades, problems involving $p(\cdot)$-Laplacian and fractional $p$-Laplacian have been studied intensively. These topics has become the center of studying PDEs because of its mathematical challenges and real applications. Very recently, elliptic problems involving
the fractional $p(\cdot)$-Laplacian has been investigated. The solution space for such problems is the fractional Sobolev spaces with variable exponent. However, such spaces have not been well-defined as well as not many properties on such spaces have been established. In this paper we first refine the fractional Sobolev space with variable exponent investigated in \cite{Kaufmann,Bahrouni.DCDS2018,Bahrouni.JMAA2018} and obtain fundamental imbeddings on our space. With these imbeddings in hand, we investigate the boundedness and multiplicity of solutions to the following problem
\begin{equation}\label{eq}
	\begin{cases}
		(-\Delta)_{p(x)}^su=f(x,u) \quad \text{in} \quad \Omega, \\
		u=0 \quad \text{in} \quad \Bbb R^N\setminus \Omega,
	\end{cases}
\end{equation}
where $\Omega$ is a bounded Lipschitz domain in $\mathbb{R}^N$ ($N\ge 2$); $s\in (0,1)$;  $p(x)=\widetilde{p}(x,x)$ for all $x\in\ol{\Omega}$ with $\widetilde{p}\in C(\mathbb{R}^N\times\mathbb{R}^N)$ satisfying $\widetilde{p}(x,y)=\widetilde{p}(y,x)$ for all $x,y\in\mathbb{R}^N$ and $1<\inf_{(x,y)\in \mathbb{R}^N\times\mathbb{R}^N}\widetilde{p}(x,y)\leq \sup_{(x,y)\in\mathbb{R}^N\times\mathbb{R}^N}\widetilde{p}(x,y)<\frac{N}{s};$ the operator $(-\Delta)_{p(\cdot)}^s$ is defined as
\begin{equation*}
	(- \Delta)_{p(x)}^s\, u(x) = 2\ \lim_{\varepsilon \searrow 0} \int_{\mathbb{R}^N \setminus B_\varepsilon(x)} \frac{|u(x) - u(y)|^{\widetilde{p}(x,y)-2}\, (u(x) - u(y))}{|x - y|^{N+s\widetilde{p}(x,y)}}\, \diff y, \quad x \in \mathbb{R}^N,
\end{equation*}
where $B_\epsilon(x):=\{z\in\mathbb{R}^N: |z-x|<\epsilon\};$ and $f:\ \Omega\times\Bbb R \to \Bbb R$ is a Carath\'eodory function.

\smallskip
It is a natural question whether the classical results can be recovered when elliptic equations involving the $p$-Laplacian (or $p(\cdot)$-Laplacian) are changed into non-local variational problems with variable exponents. Very recently, U. Kaufmann {\it et al.} \cite{Kaufmann} first introduced new class of fractional Sobolev spaces with variable exponents $W^{s,q(\cdot),\widetilde{p}(\cdot,\cdot)}(\Omega)$ defined as
$$W^{s,q(\cdot),\widetilde{p}(\cdot,\cdot)}(\Omega):=\bigg\{ u \in L^{q(\cdot)}(\Omega):
\int_{\Omega}\int_{\Omega} \frac{|u(x)-u(y)|^{\widetilde{p}(x,y)}}{|x-y|^{N+s\widetilde{p}(x,y)}} \,\diff x \diff y < +\infty\bigg\}, $$
where $q\in C(\overline{\Omega},(1,\infty)).$ With the restriction $\widetilde{p}(x,x)<q(x)$ for all $x\in\overline{\Omega},$ they obtained the compact imbedding $W^{s,q(\cdot),\widetilde{p}(\cdot,\cdot)}(\Omega)\hookrightarrow\hookrightarrow L^{r(\cdot)}(\Omega)$ for any $r\in C(\overline{\Omega})$ satisfying $1<r(x)<\frac{N\widetilde{p}(x,x)}{N-s\widetilde{p}(x,x)}$ for all $x\in \overline{\Omega}$. With this compact imbedding result in hand, the authors in \cite{Kaufmann} obtained a simple existence result by applying a direct  method of Calculus of Variations for the energy functional of the form
$$\mathcal{F}(u)=\int_{\Omega}\int_{\Omega} \frac{|u(x)-u(y)|^{\widetilde{p}(x,y)}}{\widetilde{p}(x,y)|x-y|^{N+s\widetilde{p}(x,y)}} \,\diff x \diff y+\int_{\Omega}\frac{|u(x)|^{q(x)}}{q(x)}\diff x-\int_{\Omega}f(x)u(x)\diff x.$$
The authors in \cite{Bahrouni.DCDS2018} gave some further basic properties both on this function space and the related nonlocal operator. As applications, they investigated the existence of solutions for
\eqref{eq} in the case of the prototype $f(x,u)=\lambda|u|^{r(x)-1}u-|u|^{q(x)-1}u$ where $\lambda>0$, $1< r(x)<\inf_{x\in\Omega}p(x)<q(x)$ for all $x\in\ol{\Omega}$. Motivated by the papers \cite{Bahrouni.DCDS2018,Kaufmann}, a variant of comparison principle for the fractional $p(\cdot)$-Laplacian and sub--supersolution principle for \eqref{eq} was presented in \cite{Bahrouni.JMAA2018}.

\smallskip
The first aim of our paper to get rid of the restricted condition $\widetilde{p}(x,x)<q(x)$ for all $x\in\overline{\Omega}$ for the compact imbedding $W^{s,q(\cdot),\widetilde{p}(\cdot,\cdot)}(\Omega)\hookrightarrow\hookrightarrow L^{r(\cdot)}(\Omega)$. Obviously, under this condition the space $W^{s,q(\cdot),\widetilde{p}(\cdot,\cdot)}(\Omega)$ is acctually not a generalization of the usual fractional Sobolev space $W^{s,p}(\Omega)$ with constant exponent. Furthermore, we also obtain the continuous imbeddings when the domain is the whole space $\mathbb{R}^N.$

Our next aim to provide a sufficient conditions guaranteeing global a-priori bounds for weak solutions of problem \eqref{eq}. The main tools for obtaining this result are the De Giorgi's iteration and a localization method. This approach originally comes from the paper \cite{VZ}. Inspired by \cite{VZ}, the boundedness of weak solutions for elliptic equations with a variable exponents and nonlinear conormal derivative boundary condition has been investigated in \cite{WZ}; see also \cite{CKK}. By modifying the techniques used in \cite{WZ}, K. Ho and I. Sim \cite{HS2014} dealt with degenerated $p(\cdot)$-Laplace equations of the form
\begin{align*}
	\begin{cases}
		-\text{div}(w(x)|\nabla u|^{p(x)-2}\nabla u)=f(x,u) \quad &\text{in} \quad \Omega, \\
		u=0 \quad &\text{on} \quad \partial \Omega.
	\end{cases}
\end{align*}
A natural question is to know whether these global a priori bounds hold for the fractional $p(\cdot)$-Laplacian.
As compared with elliptic equations involving the $p(\cdot)$-Laplacian, the value of $(-\Delta)^s_{p(x)} u(x)$ at any point $x\in \Omega$ relies not only on the values of $u$ and $p(\cdot)$ on the whole $\Omega$, but actually on the entire space $\Bbb R^N$. In this regard, more complicated analysis than the papers \cite{CKK,VZ,HS2014} has to be carefully carried out. To the best of the authors' knowledge, the present paper seems to be the first to study the regularity of weak solutions to the fractional $p(\cdot)$-Laplacian problems. Especially, even if we use the De Giorgi iteration and a localization method, it is noteworthy that we provide new condition for the exponent $p(\cdot,\cdot)$ on $\mathbb{R}^N\times\mathbb{R}^N$ in order to investigate the $L^{\infty}$-bound of weak solutions to \eqref{eq}.

\smallskip
In recent years, existence of infinitely many solutions to equations of the elliptic type has attracted much attention and has been extensively studied in the literature; see for example \cite{AL,ZBS15,CKK,Gu,Kim,KKL,KKL19,N,W0,Za} and the references therein. As an application of the $L^{\infty}$--boundedness of weak solutions, we obtain that
nonlinear problems associated with the fractional $p(\cdot)$-Laplacian admit a sequence of infinitely many small energy solutions whose their $L^{\infty}$-norms converge to zero. This existence result to nonlinear boundary value problems
\begin{align*}
	\begin{cases}
		-\Delta u=\lambda\abs{u}^{r-1}u + f(x,u)  \quad &\textmd{in} \ \
		\Omega,\\
		u= 0\quad &\text{on}\ \ \mathbb{R}^N\backslash\Omega,
	\end{cases}
\end{align*}
was originally observed by Z.-Q. Wang \cite{W0} where $0<r<1$, and the nonlinear term $f$ was considered as a perturbation term. In order to obtain this existence result, he divided the outlines of the proof into two steps. The first one is to utilize the modified functional method. More precisely, he modified and
extended the function $f(x, u)$ to an adequate function ${\widetilde f} (x,
u)$ to apply global variational formulation in \cite{H}. The other one is to get the existence of a sequence of solutions converging to $0$ in $L^{\infty}$ by applying the standard regularity theory. Employing this argument in \cite{W0}, Z. Guo \cite{Gu} showed that the $p$-Laplacian equations with
indefinite concave nonlinearities have infinitely many solutions; see also \cite{CKK, Kim, N}. As we know, some global assumptions on $f(x,u)$ were used in an essential way to derive the existence of infinitely many solutions for elliptic equations; see \cite{AL,ZBS15,CKK,KKL,KKL19,Za}. However the authors in \cite{CKK, Gu, Kim, N, W0} investigated the existence of small energy solutions to equations of the elliptic type when the conditions on $f(x,
u)$ are imposed near zero; in particular, $f(x,u)$ is odd in $t$ for
a small $t$, and no conditions on $f (x, u)$ exist at infinity. In particular, if we apply the well known regularity arguments for elliptic equations, it is easy to establish the existence of such a sequence of solutions belonging to $L^{\infty}$ space. As far as we are aware, there were no such regularity and existence results for fractional $p(\cdot)$-Laplacian problems. In comparison with the papers \cite{CKK, Gu, W0}, the main difficulty to obtain our second aim is to show the $L^{\infty}$-bound of weak solutions for the given problem. We remark that the strategy for obtaining this multiplicity is to assign a regularity-type result in our second aim.

One of the novelties of this paper is that we refine the fractional Sobolev spaces with variable exponents given in \cite{Kaufmann,Bahrouni.DCDS2018,Bahrouni.JMAA2018} and obtain fundamental imbeddings on our new space. The other one is to give a sufficient condition for the exponent $p(\cdot,\cdot)$ on $\mathbb{R}^N\times\mathbb{R}^N$ for that achieves the iteration argument of De Giorgi type and get global boundedness of weak solutions to \eqref{eq}.

This paper is organized as follows. In Section 2, we briefly review the definitions and collect some preliminary results for the Lebesgue spaces with variable exponent and the fractional Sobolev spaces. In Section 3, we refine fractional Sobolev space with variable exponent given in \cite{Kaufmann} and obtain the crucial imbeddings from these spaces into variable exponent Lebesgue spaces. The main result that requires the new condition for the variable exponent $p(\cdot,\cdot)$ on $\mathbb{R}^N\times\mathbb{R}^N$ is proven in Section 4. For this we employ the De Giorgi's iteration and a localization method, which is suitable to investigate the $L^{\infty}$-bound of weak solutions to \eqref{eq}. As its application to the fractional $p(\cdot)$-Laplacian problems, Section 5 gives the existence of a sequence of infinitely many small energy solutions whose their $L^{\infty}$-norms converge to zero.

\section{Variable exponent Lebesgue spaces and fractional Sobolev spaces}
In this section, we briefly review the definitions and list some basic properties of the Lebesgue spaces with variable exponent and the fractional Sobolev spaces.

Let $\Omega$ be a Lipschitz domain in $\mathbb{R}^N.$ Denote
$$
C_+(\overline\Omega)=\left\{h\in C(\overline\Omega):
1<\inf_{x\in\overline\Omega}h(x)\leq \sup_{x\in\overline\Omega}h(x)<\infty\right\},
$$
and for $h\in C_+(\overline\Omega),$ denote
$$
h^+=\sup\limits_{x\in \overline\Omega}h(x)\quad \hbox{and}\quad
h^-=\inf\limits_{x\in \overline\Omega}h(x).
$$
For $p\in C_+(\overline\Omega)$ and a $\sigma$-finite, complete measure $\mu$ in $\Omega,$ define the variable
exponent Lebesgue space
$$
L_\mu^{p(\cdot)}(\Omega) := \left \{ u : \Omega\to\mathbb{R}\  \hbox{is}\  \mu-\text{measurable},\ \int_\Omega |u(x)|^{p(x)} \;\diff\mu < \infty \right \},
$$
endowed with the Luxemburg norm
$$
\norm{u}_{L_\mu^{p(\cdot)}(\Omega)}:=\inf\left\{\lambda >0:
\int_\Omega
\Big|\frac{u(x)}{\lambda}\Big|^{p(x)}\;\diff\mu\le1\right\}.
$$
 When $\diff \mu=\diff x$ the usual Lebesgue measure, we write  $L^{p(\cdot) }(\Omega) $  and $\norm{u}_{L^{p(\cdot)}(\Omega)}$  instead of $L_\mu^{p(\cdot)}(\Omega)$ and $\norm{u}_{L_\mu^{p(\cdot)}(\Omega)}$, respectively. Some basic properties of $L_\mu^{p(\cdot)}(\Omega)$ are listed in the next three propositions.

\begin{proposition}{\rm (\cite[Corollary 3.3.4]{DHHR})}\label{est.Deining} Let $\al,\beta\in C_+(\ol{\Omega})$ such that $\al(x)\le \beta(x)$ for all $x\in\ol{\Omega}.$ Then, we have 
$$\norm{u}_{L_\mu^{\alpha(\cdot)}(\Omega)}\le
	2\big[1+\mu(\Omega)\big]\norm{u}_{L_\mu^{\beta(\cdot)}(\Omega)},\forall u\in L_\mu^{\alpha(\cdot)}(\Omega)\cap L_\mu^{\beta(\cdot)}(\Omega).$$
\end{proposition}

 \begin{proposition} \label{lem1}{\rm (\cite{FZ, KR})}
The space $L^{p(\cdot)}(\Omega)$ is a separable, uniformly
convex Banach space, and its dual space is
$L^{p'(\cdot)}(\Omega),$ where $1/p(x)+1/p'(x)=1$. For any $u\in
L^{p(\cdot)}(\Omega)$ and $v\in L^{p'(x)}(\Omega)$, we have
$$
\Big|\int_\Omega uv\;\diff x\Big|
\le
2\norm{u}_{L^{p(\cdot)}(\Omega)}\norm{v}_{L^{p'(\cdot)}(\Omega)}.
$$
\end{proposition}

\begin{proposition} \label{norm-modular}
{\rm (\cite{FZ})} Define the modular $\rho:\ L^{p(\cdot)}(\Omega)\to\mathbb{R}$ as
$$
\rho(u):=\int_\Omega |u|^{p(x)}\;\diff x, \quad \forall u\in
L^{p(\cdot)}(\Omega).
$$
Then, we have the following relations between norm and modular.
\begin{enumerate}
	\item[\rm{(i)}] For $u\in L^{p(\cdot)}(\Omega)\setminus\{0\},$  $\lambda=\norm{u}_{L^{p(\cdot)}(\Omega)}$\ if and only if \ $\rho(\frac{u}{\lambda})=1.$
\item[\rm{(ii)}] $\rho(u)>1$ $(=1;\ <1)$ if and only if \ $\norm{u}_{L^{p(\cdot)}(\Omega)}>1$ $(=1;\ <1)$,
respectively.
\item[\rm{(iii)}] If $\norm{u}_{L^{p(\cdot)}(\Omega)}>1$, then $
\norm{u}_{L^{p(\cdot)}(\Omega)}^{p^-}\le \rho(u)\le
\norm{u}_{L^{p(\cdot)}(\Omega)}^{p^+}$.
\item[\rm{(iv)}] If $\norm{u}_{L^{p(\cdot)}(\Omega)}<1$, then $
\norm{u}_{L^{p(\cdot)}(\Omega)}^{p^+}\le \rho(u)\le
\norm{u}_{L^{p(\cdot)}(\Omega)}^{p^-}$.
\end{enumerate}
\end{proposition}

Let $s\in (0,1)$ and $p\in (1,\infty)$ be constants. Define the fractional Sobolev space $
W^{s,p}(\Omega)$ as
$$
W^{s,p}(\Omega):=\big\{u\in L^p(\Omega): \int_{\Omega}\int_{\Omega}\frac{|u(x)-u(y)|^p}{|x-y|^{N+sp}}\diff x\diff y<\infty\big\}
$$
endowed with norm
\begin{equation*}
\label{norm}
\|u\|_{s,p,\Omega}:=\left(\int_{\Omega}|u(x)|^p\diff x\right)^{1/p}+\left(\int_{\Omega}\int_{\Omega}\frac{|u(x)-u(y)|^p}{|x-y|^{N+sp}}\diff x \diff y\right)^{1/p}.
\end{equation*}
We recall the following crucial results.
\begin{proposition}{\rm (\cite{DPV})} \label{imb.frac.const}Let $s\in (0,1)$ and $p\in (1,\infty)$ be such that $sp<N.$ It holds that
	\begin{itemize}
		\item[(i)] $W^{s,p}(\Omega)\hookrightarrow \hookrightarrow L^{q}(\Omega)$ if $\Omega$ is bounded and $1\leq q< p_s^\ast:=\frac{Np}{N-sp}$;
		\item[(i)]$W^{s,p}(\Omega)\hookrightarrow  L^{q}(\Omega)$ if $p\leq q\leq p_s^\ast.$
	\end{itemize}
\end{proposition}
\begin{proposition}{\rm (\cite{DPV})} \label{P-S.inequality}
	Let $s\in (0,1)$ and $p\in (1,\infty)$ be such that $sp<N.$ Then, there exists a positive constant $C=C(N,p,s)$ such that, for any measurable and compactly supported function $f:\mathbb{R}^N\to\mathbb{R},$ we have
	$$\|f\|_{L^{p_s^\ast}(\mathbb{R}^N)}\leq C\left(\int_{\mathbb{R}^N}\int_{\mathbb{R}^N}\frac{|f(x)-f(y)|^p}{|x-y|^{N+sp}}\diff x \diff y\right)^{1/p}.$$
\end{proposition}
\section{The Sobolev spaces $W^{s,p(\cdot,\cdot)}(\Omega)$}
In this section, we refine the definition and some imbedding results on fractional Sobolev spaces with variable exponent that was first introduced in \cite{Kaufmann}.

 Let $\Omega$ be a Lipschitz domain in $\mathbb{R}^N$. Let $0<s<1$ and let $p\in C(\ol{\Omega}\times
\ol{\Omega})$ be such that $p$ is symmetric, i.e.,
$p(x,y) = p(y, x)$ for all $x,y\in \ol{\Omega},$ and $$1 < p^-:=\inf_{(x,y)\in \ol{\Omega}\times
\ol{\Omega}}p(x,y) \le p^+:=\sup_{(x,y)\in \ol{\Omega}\times
\ol{\Omega}}p(x,y) < +\infty.$$

For $q
\in C_+(\ol{\Omega})$, define
$$
W^{s,q(\cdot),p(\cdot,\cdot)}(\Omega):=\bigg\{ u \in L^{q(\cdot)}(\Omega):
\int_{\Omega}\int_{\Omega}
\frac{|u(x)-u(y)|^{p(x,y)}}{|x-y|^{N+sp(x,y)}} \,\diff x \diff y < +
\infty\bigg\},
$$
and for $u\in W^{s,q(\cdot),p(\cdot,\cdot)}(\Omega),$ set
$$
[u]_{s,p(\cdot,\cdot),\Omega}:=\inf \left \{\lambda>0:
\int_{\Omega}\int_{\Omega}
\frac{|u(x)-u(y)|^{p(x,y)}}{\lambda^{p(x,y)}|x-y|^{N+sp(x,y)}}
\,\diff x \diff y <1 \right \}.
$$
Then $W^{s,q(\cdot),p(\cdot,\cdot)}(\Omega)$ endowed with the norm
$$
\norm{u}_{s,q,p,\Omega}:=\norm{u}_{L^{q(\cdot)}(\Omega)}+[u]_{s,p(\cdot,\cdot),\Omega}
$$
is a separable reflexive Banach space (see \cite{Kaufmann,Bahrouni.DCDS2018, Bahrouni.JMAA2018}). On $W^{s,q(\cdot),p(\cdot,\cdot)}(\Omega)$ we shall sometimes work with the norm
$$
\abs{u}_{s,q,p,\Omega}:=\inf \left \{\lambda>0: \widetilde{\rho}\left(\frac{u}{\lambda}\right) <1 \right \},
$$
where $\widetilde{\rho}(u):=\int_{\Omega}\left|u\right|^{q(x)}\diff x+
\int_{\Omega}\int_{\Omega}
\frac{|u(x)-u(y)|^{p(x,y)}}{|x-y|^{N+sp(x,y)}}
\,\diff x\diff y.$  It is not difficult to see that $\abs{\cdot}_{s,q,p,\Omega}$ is an equivalent norm of $\norm{\cdot}_{s,q,p,\Omega}$ with the relation
\begin{equation}\label{equivalent.norms}
\frac{1}{2}\norm{u}_{s,q,p,\Omega}\leq \abs{u}_{s,q,p,\Omega}\leq 2\norm{u}_{s,q,p,\Omega}.
\end{equation}
In what follows, for brevity, in some places we write $p(x)$ instead of $p(x,x)$ and in this sense, $p\in C_+(\ol{\Omega})$. Also, we write $W^{s,p(\cdot,\cdot)}(\Omega)$ instead of $W^{s,p(\cdot),p(\cdot,\cdot)}(\Omega)$ and when the domain $\Omega$ is understood, we just write $\norm{u}_{s,p}$ and $|u|_{s,p}$ instead of $\norm{u}_{s,p,\Omega}$ and $|u|_{s,p,\Omega}$, respectively. The following relations between the norm  $\abs{\cdot}_{s,q,p,\Omega}$ and the modular $\widetilde{\rho}(\cdot)$ can be easily obtained from their definitions.
\begin{proposition} \label{norm-modular2} On $W^{s,q(\cdot),p(\cdot,\cdot)}(\Omega)$ it holds that
		\begin{enumerate}
		\item[\rm{(i)}] for $u\in W^{s,q(\cdot),p(\cdot,\cdot)}(\Omega)\setminus\{0\},$  $\lambda=|u|_{s,q,p,\Omega}$\ if and only if \ $\widetilde{\rho}(\frac{u}{\lambda})=1;$
		\item[\rm{(ii)}] $\widetilde{\rho}(u)>1$ $(=1;\ <1)$ if and only if \ $|u|_{s,q,p,\Omega}>1$ $(=1;\ <1)$,
		respectively.
			\end{enumerate}
Moreover, on $W^{s,p(\cdot,\cdot)}(\Omega)$ it holds that 	
	\begin{enumerate}	
		\item[\rm{(iii)}] if $|u|_{s,p}\geq 1$, then $
		|u|_{s,p}^{p^-}\le \widetilde{\rho}(u)\le
		|u|_{s,p}^{p^+}$;
		\item[\rm{(iv)}] if $|u|_{s,p}<1$, then $
		|u|_{s,p}^{p^+}\le \widetilde{\rho}(u)\le
		|u|_{s,p}^{p^-}$.
	\end{enumerate}
\end{proposition}

Our first main result in this section is the next theorem, which refines the result obtained in \cite[Theorem 1.1]{Kaufmann} for a bounded domain $\Omega$.
\begin{theorem}\label{general.imb}
Let $\Omega$ be a bounded Lipschitz domain and let $p, q$ and $s$ be as above. Assume furthermore that $$
sp^+<N \ \ \text{and} \ \ q(x)\ge p(x,x),\quad  \forall x\in \ol\Omega.
$$
Then, it holds that
$$
W^{s,q(\cdot),p(\cdot,\cdot)}(\Omega) \hookrightarrow \hookrightarrow
L^{r(\cdot)}(\Omega)$$ for any $r\in C_+(\ol{\Omega})$ such that
$r(x)<p^*_s(x):=\frac{Np(x,x)}{N-sp(x,x)}$ for all $x\in \ol\Omega$.
\end{theorem}

\begin{remark}\rm
It is worth pointing out that in existing articles \cite{Kaufmann,Bahrouni.DCDS2018,Bahrouni.JMAA2018} working on $
W^{s,q(\cdot),p(\cdot,\cdot)}(\Omega)$, the function $q$ is actually assumed that $q(x)>p(x,x),$
for all $x\in \ol\Omega$ due to some technical reason. Such spaces are actually not a generalization of the fractional Sobolev space $W^{s,p}(\Omega)$. Our result therefore is an 
improvement of \cite[Theorem 1.1]{Kaufmann}.
\end{remark}

\vspace{0.3cm}
The following proof of Theorem~\ref{general.imb} is based on the idea used in \cite[Proof of Theorem 1.1]{Kaufmann}.
In what follows, denote by $B_\epsilon(x_0)$ the open ball centered at $x_0$ with radius $\epsilon$ in an appropriate Euclidean space $\mathbb{R}^k$ and denote by $|S|$ the Lebesgue measure of $S\subset \mathbb{R}^N.$

\begin{proof}[Proof of Theorem~\ref{general.imb}]
Since $p,q$ and $r$ are continuous on the compact set $\overline{\Omega}$, we have
\begin{equation*}\label{11111}
\alpha:=\min_{x\in\overline{\Omega}}\left[\frac{Np(x,x)}{N-sp(x,x)}-r(x)\right]>0.
\end{equation*}
By the boundedness and the Lipschitz property of $\Omega$, for any given $\epsilon>0$ small enough, we can cover $\overline{\Omega}$ by a finite of balls $\{B_i\}_{i=1}^{m}$ with radius $\epsilon$ such that $\Omega_i:=B_i\cap\Omega$ ($i=1,\cdots,m$) are Lipschitz domains as well. By the continuity of
$p,q$ and $r$ again, we can choose $t\in (0,s)$ and $\e>0$ such that
\begin{equation}\label{subscitical.loc}
p_{t,i}^*:=\frac{Np_i}{N-tp_i}\ge r_i^++\frac{\al}{2}
\end{equation}
and
\begin{equation}\label{iiiiii}
q(x)\ge p(x,x)\ge p_i, \ \forall x\in \Omega_i,
\end{equation}
for all $i\in\{1,\cdots,m\},$ where $p_i:=\inf_{(z,y)\in \Omega_i\times \Omega_i}p(z,y)$ and $r^+_i:=\sup_{x\in \Omega_i}r(x)$.

Let $\{\xi_i\}_{i=1}^m$ be a partition of unity of $\overline{\Omega}$ associated with the covering $\{B_i\}_{i=1}^{m}$ of $\overline{\Omega},$ i.e., for each $i\in\{1,\cdots,m\}$, $\xi_i\in C^\infty(\mathbb{R}^N)$, $0\leq \xi_i \leq 1$, $\operatorname{supp} (\xi_i)\subset B_i,$ and
 $$\sum_{i=1}^{m}\xi_i=1\quad \text{on}\quad \overline{\Omega}.$$
We claim that for all $u\in W^{s,q(\cdot),p(\cdot,\cdot)}(\Omega)$ and all $i\in\{1,\cdots,m\}$, $\xi_iu\in W^{t,p_i}(\Omega),$ i.e.,
\begin{equation}\label{Wt,pi.1}
\int_{\Omega}|\xi_{i}u|^{p_i}\diff x+\int_{\Omega}\int_{\Omega}
\frac{|\xi_i(x)u(x)-\xi_i(y)u(y)|^{p_i}}{|x-y|^{N+tp_i}} \,\diff x \diff y<\infty.
\end{equation}
Indeed, using \eqref{iiiiii} we estimate the first integral in \eqref{Wt,pi.1} as follows:
\begin{align*}\label{Wt,pi.2}
\int_{\Omega}|\xi_{i}u|^{p_i}\diff x=\int_{\Omega_i}|\xi_{i}u|^{p_i}\diff x\leq \int_{\Omega_i}\left(|\xi_{i}u|^{q(x)}+1\right)\diff x\leq \int_{\Omega_i}|u|^{q(x)}\diff x+|\Omega_i|<\infty.
\end{align*}
Then \eqref{Wt,pi.1} follows if we can show that
\begin{equation}\label{Wt,pi.3}
[\xi_iu]_{t,p_i,\Omega}\leq  C_1\round{[u]_{t,p_i,\Omega_i}+\|u\|_{L^{q(\cdot)}(\Omega)}}
	\end{equation}
	and
	\begin{equation}\label{Wt,pi.4}
	[u]_{t,p_i,\Omega_i}\leq  C_2[u]_{s,p(\cdot,\cdot),\Omega},
	\end{equation}
where $C_1$ and $C_2$ are positive constants independent of $u$. To show \eqref{Wt,pi.3}, we first note that
\begin{align*}
&\int_{\Omega}\int_{\Omega}
\frac{|\xi_i(x)u(x)-\xi_i(y)u(y)|^{p_i}}{|x-y|^{N+tp_i}} \,\diff x \diff y\notag\\
&=\int_{\Omega_i}\int_{\Omega_i}
\frac{|\xi_i(x)u(x)-\xi_i(y)u(y)|^{p_i}}{|x-y|^{N+tp_i}} \,\diff x \diff y +2\int_{\Omega\setminus\Omega_i}\int_{\Omega_i}
\frac{|\xi_i(x)u(x)|^{p_i}}{|x-y|^{N+tp_i}} \,\diff x \diff y\notag\\
&\leq2^{p_i}\int_{\Omega_i}\int_{\Omega_i}
\frac{|\xi_i(x)|^{p_i}|u(x)-u(y)|^{p_i}}{|x-y|^{N+tp_i}} \,\diff x \diff y+2^{p_i}\int_{\Omega_i}\int_{\Omega_i}
\frac{|\xi_i(x)-\xi_i(y)|^{p_i}}{|x-y|^{N+tp_i}}|u(y)|^{p_i} \,\diff x \diff y\notag\\
&\quad\quad\quad\quad\quad\quad +2\int_{\Omega_i}\int_{\Omega\setminus\Omega_i}
\frac{|\xi_i(x)-\xi_i(y)|^{p_i}}{|x-y|^{N+tp_i}}| u(x)|^{p_i}\, \diff y\diff x.
\end{align*}
Using the facts that $0\leq \xi_i(x)\leq 1$ and $|\xi_i(x)-\xi_i(y)|\leq \|\nabla\xi_i\|_\infty |x-y|$ for all $x,y\in\Omega$, we deduce from the last inequality that
\begin{align*}
&\int_{\Omega}\int_{\Omega}
\frac{|\xi_i(x)u(x)-\xi_i(y)u(y)|^{p_i}}{|x-y|^{N+tp_i}} \,\diff x \diff y\notag\\
&\leq2^{p_i}\int_{\Omega_i}\int_{\Omega_i}
\frac{|u(x)-u(y)|^{p_i}}{|x-y|^{N+tp_i}} \,\diff x \diff y+2^{p_i}\|\nabla\xi_i\|_\infty^{p_i}\int_{\Omega_i}\left(\int_{\Omega_i}
|x-y|^{-N+(1-t)p_i} \,\diff x\right) |u(y)|^{p_i} \diff y\notag\\
&\quad +2\|\nabla\xi_i\|_\infty^{p_i}\int_{\Omega_i}\left(\int_{\Omega\setminus\Omega_i}
|x-y|^{-N+(1-t)p_i} \,\diff y\right)|u(x)|^{p_i}\, \diff x\notag\\
&\leq2^{p_i}\int_{\Omega_i}\int_{\Omega_i}
\frac{|u(x)-u(y)|^{p_i}}{|x-y|^{N+tp_i}} \,\diff x \diff y+2^{p_i}\|\nabla\xi_i\|_\infty^{p_i}\int_{\Omega_i}\left(\int_{\Omega}
|x-y|^{-N+(1-t)p_i} \,\diff y\right) |u(x)|^{p_i} \diff x.
\end{align*}
Note that
\begin{equation*}
\sup_{x\in\Omega}\ \int_{\Omega}
|x-y|^{-N+(1-t)p_i} \,\diff y\leq \int_{B(0,R)}
|z|^{-N+(1-t)p_i}\diff z\leq \frac{\omega_N R^{(1-t)p_i}}{(1-t)p_i},
\end{equation*}
where $B(0,R)\supset \Omega-\Omega,$ and $\omega_N$ denotes the area of the unit sphere in $\mathbb{R}^N.$  On the other hand, applying Proposition~\ref{est.Deining} with noting $p_i\leq q(x)$ for all $x\in\Omega_i$ we have
\begin{equation}\label{Wt,pi.4'}
\norm{u}_{L^{p_{i}}(\Omega_i)}\leq 2\left(1+|\Omega_i|\right)\norm{u}_{L^{q(\cdot)}(\Omega_i)}\leq 2\left(1+|\Omega|\right)\norm{u}_{L^{q(\cdot)}(\Omega)}.
\end{equation}
From the last three inequalities, we obtain \eqref{Wt,pi.3}.

To obtain \eqref{Wt,pi.4}, we consider the measure $\mu$ in $\mathbb{R}^N\times \mathbb{R}^N$ with $\diff \mu (x,y):=\frac{\diff x\diff y}{|x-y|^{N+(t-s)p_i}}$ and set
$$ F(x,y):=\frac{|u(x)-u(y)|}{|x-y|^s},\quad \forall x,y\in\mathbb{R}^N,\ x\ne y.$$
Invoking Proposition~\ref{est.Deining}, we have
\begin{align}\label{Wt,pi.5}
[u]_{t,p_i,\Omega_i}&=\round{\int_{\Omega_i}\int_{\Omega_i}
	\frac{|u(x)-u(y)|^{p_i}}{|x-y|^{N+tp_i+sp_i-sp_i}}
	\,\diff x\diff y}^{\frac{1}{p_i}}\notag\\
&=\round{\int_{\Omega_i}\int_{\Omega_i} \round{\frac{|u(x)-u(y)|}{|x-y|^{s}}}^{p_i}
	\,\frac{\diff x\diff y}{|x-y|^{N+(t-s)p_i}}}^{\frac{1}{p_i}}\notag\\
&=\norm{F}_{L_\mu^{p_i}(\Omega_i\times \Omega_i)}\notag\\
&\le 2\left[1+\mu(\Omega_i\times \Omega_i)\right]\norm{F}_{L_\mu^{p(\cdot,\cdot)}(\Omega_i\times \Omega_i)}.
\end{align}
Hence, to obtain \eqref{Wt,pi.4} it suffices to prove that
\begin{equation}\label{Wt,pi.6}
\norm{F}_{L_\mu^{p(\cdot,\cdot)}(\Omega_i\times \Omega_i)}\le
K[u]_{s,p(\cdot,\cdot),\Omega},
\end{equation}
where $K:=\max\left\{1,\sup_{(x,y)\in \Omega\times \Omega}|x-y|^{s-t}\right\}\in [1,\infty)$. To this end, let $\la>0$ be such that
\begin{equation}\label{66666}
\int_{\Omega}\int_{\Omega}
\frac{|u(x)-u(y)|^{p(x,y)}}{\la^{p(x,y)}|x-y|^{N+sp(x,y)}} \,\diff x\diff y<1.
\end{equation}
Then, for $\widetilde{\la}:=K\la$ we have
\begin{align*}
&\int_{\Omega_i}\int_{\Omega_i}
\round{\frac{|u(x)-u(y)|}{\widetilde{\la}|x-y|^{s}}}^{p(x,y)}
\,\diff \mu (x,y)\\
&=\int_{\Omega_i}\int_{\Omega_i}\frac{|x-y|^{(s-t)p_i}}{K^{p(x,y)}}
\frac{|u(x)-u(y)|^{p(x,y)}}{\la^{p(x,y)}|x-y|^{N+sp(x,y)}}
\,\diff x\diff y\\
&\le\int_{\Omega_i}\int_{\Omega_i}
\frac{|u(x)-u(y)|^{p(x,y)}}{\la^{p(x,y)}|x-y|^{N+sp(x,y)}}
\,\diff x\diff y < 1.
\end{align*}
Thus,
\begin{equation}\label{77777}
\norm{F}_{L_\mu^{p(\cdot,\cdot)}(\Omega_i\times \Omega_i)}\le \widetilde{\la}=K\la .
\end{equation}
Taking infimum over all $\la>0$ satisfying \eqref{66666}, we get \eqref{Wt,pi.6}
from \eqref{77777}. That is, we have proved \eqref{Wt,pi.3} and \eqref{Wt,pi.4}, and therefore,  $\xi_iu\in W^{t,p_i}(\Omega)$ for all $u\in W^{s,q(\cdot),p(\cdot,\cdot)}(\Omega)$ and all $i\in\{1,\cdots,m\}$. Hence, for each $i\in\{1,\cdots,m\}$, there exists a constant $C_{imb}^{(i)}=C_{imb}^{(i)}(N,p,t,\e,B_i)$ such that
\begin{equation}\label{33333}
\norm{\xi_iu}_{L^{p_{t,i}^*}(\Omega)}\le
C_{imb}^{(i)}\round{\norm{\xi_iu}_{L^{p_{i}}(\Omega)}+[\xi_iu]_{t,p_i,\Omega}},\quad \forall u\in W^{s,q(\cdot),p(\cdot,\cdot)}(\Omega)
\end{equation}
in view of Proposition~\ref{imb.frac.const}. We also note that $p_i\leq q(x)$ for all $x\in\Omega_i$, and hence, we have
\begin{equation}\label{33334}
\norm{\xi_iu}_{L^{p_{i}}(\Omega)}=\norm{\xi_iu}_{L^{p_{i}}(\Omega_i)}\leq 2\left(1+|\Omega_i|\right)\norm{\xi_iu}_{L^{q(\cdot)}(\Omega_i)}\leq 2\left(1+|\Omega|\right)\norm{u}_{L^{q(\cdot)}(\Omega)}
\end{equation}
in view of Proposition~\ref{est.Deining}.

We are now in a position to prove the imbedding $
W^{s,q(\cdot),p(\cdot,\cdot)}(\Omega) \hookrightarrow
L^{r(\cdot)}(\Omega).$ That is, we shall prove that there
exists $C>0$ such that
\begin{equation}\label{44444}
\norm{u}_{L^{r(\cdot)}(\Omega)}\le C\norm{u}_{s,q,p,\Omega}, \quad \forall u\in W^{s,q(\cdot),p(\cdot,\cdot)}(\Omega).
\end{equation}
In order to do this, we prove that there exist $C_3,C_4>0$ such that
\begin{equation}\label{aaaaa}
\norm{u}_{L^{r(\cdot)}(\Omega)}\leq C_3\sum_{i=1}^m\norm{\xi_iu}_{L^{p^*_{t,i}}(\Omega)}, \ \forall u\in W^{s,q(\cdot),p(\cdot,\cdot)}(\Omega)
\end{equation}
and
\begin{equation}\label{bbbbb}
\sum_{i=1}^m[\xi_iu]_{t,p_i,\Omega}\leq C_4\norm{u}_{s,q,p,\Omega}, \ \forall u\in
W^{s,q(\cdot),p(\cdot,\cdot)}(\Omega).
\end{equation}
Then, it is easy to see that \eqref{44444} follows from \eqref{33333}, \eqref{33334}, \eqref{aaaaa} and \eqref{bbbbb}.

To see \eqref{aaaaa}, we note that $ u=\sum_{i=1}^{m}\xi_i u \ \text{in}\ \Omega,
$ and hence,
\begin{equation}\label{55555}
\norm{u}_{L^{r(\cdot)}(\Omega)}\le
\sum_{i=1}^{m}\norm{\xi_i u}_{L^{r(\cdot)}(\Omega)}.
\end{equation}
Taking \eqref{subscitical.loc} and Proposition~\ref{est.Deining} into account we have
$$
\norm{\xi_iu}_{L^{r(\cdot)}(\Omega)}=\norm{\xi_iu}_{L^{r(\cdot)}(\Omega_i)}\le
2\left(1+|\Omega_i|\right)\norm{\xi_i u}_{L^{p^*_{t,i}}(\Omega_i)}.
$$
Combining this with \eqref{55555}, we deduce \eqref{aaaaa}. The inequality \eqref{bbbbb} is easily obtained from \eqref{Wt,pi.3} and \eqref{Wt,pi.4}. Thus, we have just obtained
the imbedding
$$
W^{s,q(\cdot),p(\cdot,\cdot)}(\Omega) \hookrightarrow
L^{r(\cdot)}(\Omega).
$$
Next, we show the compactness of the above imbedding. Let
$u_n\rightharpoonup 0$ in $W^{s,q(\cdot),p(\cdot,\cdot)}(\Omega)$. 
Thus $u_n\rightharpoonup 0$ in $W^{t,p_i}(\Omega_i)$ for all
$i\in\{1,\cdots , m\}$ due to $
W^{s,q(\cdot),p(\cdot,\cdot)}(\Omega) \hookrightarrow
W^{t,p_i}(\Omega_i)$  (see \eqref{Wt,pi.4} and \eqref{Wt,pi.4'}). 
Applying Proposition~\ref{imb.frac.const} with taking \eqref{subscitical.loc} into account, we deduce
$$
u_n \to 0 \quad \text{ in } \quad L^{r_i^{+}}(\Omega_i),
$$
and hence
$$
u_n \to 0 \quad \text{ in } \quad L^{r(\cdot)}(\Omega_i)\quad\text{for all}\ i\in\{1,\cdots,m\}.
$$
This yields
$$
u_n \to 0 \quad \text{ in } \quad L^{r(\cdot)}(\Omega).
$$
That is, we have proved that
$$
W^{s,q(\cdot),p(\cdot,\cdot)}(\Omega) \hookrightarrow \hookrightarrow
L^{r(\cdot)}(\Omega).
$$
The proof is complete.
\end{proof}


Since $
W^{s,q(\cdot),p(\cdot,\cdot)}(\Omega) \hookrightarrow
W^{s,p(\cdot,\cdot)}(\Omega)
$ if $\Omega$ is bounded and $p(x,x)\leq q(x)$ for all $x\in\ol{\Omega},$ the above result has the following important consequence.

\begin{corollary}\label{special.imb}
Let $\Omega$ be a bounded Lipschitz domain and let $s,p$ be as above such that $sp^+<N.$ Then, for any $r\in C_+(\ol{\Omega})$ such that $r(x)<p^*_s(x)$ for all $x\in \ol\Omega,$	
$$ W^{s,p(\cdot,\cdot)}(\Omega) \hookrightarrow \hookrightarrow L^{r(\cdot)}(\Omega).$$

\end{corollary}

When $\Omega=\mathbb{R}^N$, we have the following imbeddings.
\begin{theorem}\label{Theo.imb.RN}
	Let $s\in (0,1).$ Let $p\in C_+(\mathbb{R}^N\times \mathbb{R}^N)$ be a uniformly continuous and symmetric function such that $sp^+<N$. Then, it holds that
	\begin{itemize}
		\item [(i)] $W^{s,p(\cdot,\cdot)}(\mathbb{R}^N)\hookrightarrow L^{r(\cdot)}(\mathbb{R}^N)$  for any uniform continuous function  $r\in C_+(\mathbb{R}^N)$ satisfying $p(x,x)\leq r(x)$ for all $x\in\mathbb{R}^N$ and $\inf_{x\in\mathbb{R}^N}(p_s^\ast(x)-r(x))>0$;
		\item [(ii)] $W^{s,p(\cdot,\cdot)}(\mathbb{R}^N)\hookrightarrow\hookrightarrow L_{\loc}^{r(\cdot)}(\mathbb{R}^N)$ for any $r\in C_+(\mathbb{R}^N)$ satisfying  $r(x)< p_s^\ast(x)$ for all $x\in\mathbb{R}^N.$
	\end{itemize}
\end{theorem}
\begin{proof}
	(i) It suffices to prove for the case $\inf_{x\in\mathbb{R}^N}(r(x)-p(x,x))>0.$ Decompose $\mathbb{R}^N$ by cubes $Q_i\ (i=1,2,\cdots)$ with sides of length $\epsilon>0$ and parallel to the coordinate axes. By the uniform continuity of $p$ and $r$, we can choose $\epsilon$ sufficiently small and $t\in (0,s)$ such that
	$$p_i^-\leq r_i^-\leq r_i^+\leq (p_i^-)_t^\ast\quad \text{on each}\quad Q_i,$$
	where $p_i^-:=\inf_{(x,y)\in Q_i\times Q_i}p(x,y),\ r_i^-:=\inf_{x\in Q_i}r(x),$ and $ r_i^+:=\sup_{x\in Q_i}r(x).$ Let $u\in W^{s,p(\cdot,\cdot)}(\mathbb{R}^N)\setminus \{0\}.$ Set $v:=\frac{u}{|u|_{s,p,\mathbb{R}^N}}.$ Then, by Proposition~\ref{norm-modular2} we have
	\begin{equation}\label{P.T.v}
\int_{\mathbb{R}^N}|v|^{p(x)}\diff x+
\int_{\mathbb{R}^N}\int_{\mathbb{R}^N}
\frac{|v(x)-v(y)|^{p(x,y)}}{|x-y|^{N+sp(x,y)}}
\,\diff x\diff y =1.
	\end{equation}
From this, we have
\begin{equation}\label{P.T.norm.Qi}
|v|_{s,p,Q_i}\leq 1,\quad \forall i\in\mathbb{N}.
\end{equation}		
We claim that there exists a positive constant $C$ such that
\begin{equation}\label{P.T.loc.const}
\|v\|_{L^{r(\cdot)}(Q_i)}\leq C |v|_{s,p,Q_i},\quad \forall i\in\mathbb{N}.
\end{equation}
Here and in the rest of the proof, the positive constant $C$ may vary from line to line but depends only on $p^+,p^-,s,t,\epsilon$ and $Q_0,$ where $Q_0$ denotes the cube centered at the origin with sides of length $\epsilon>0$ and parallel to the coordinate axes. We just sketch an idea how to obtain \eqref{P.T.loc.const}. 
Invoking Proposition~\ref{est.Deining}, we have
\begin{equation}\label{prf.theo3.5.est1}
\|v\|_{L^{r(\cdot)}(Q_i)}\leq C\|v\|_{L^{r_i^+}(Q_i)}.
\end{equation}
On the other hand, arguing as in that obtained \eqref{Wt,pi.5} and \eqref{Wt,pi.6} we get
\begin{equation}\label{prf.theo3.5.est2}
\|v\|_{t,p^-_i,Q_i}\leq C\|v\|_{s,p,Q_i}\leq 2C|v|_{s,p,Q_i}.
\end{equation}
Arguing as in \cite[Proof Theorem 5.4]{DPV} and using a translation, we can find an extension $\widetilde{v}\in W^{t,p_i^-}(\mathbb{R}^N)$  with compact support in $\mathbb{R}^N$ such that $\widetilde{v}=v$ on $Q_i$, and
\begin{equation}\label{prf.theo3.5.est3}
\|\widetilde{v}\|_{t,p_i^-,\mathbb{R}^N}\leq C\|v\|_{t,p_i^-,Q_i}.
\end{equation}
Finally, we apply Proposition~\ref{P-S.inequality} for $W^{t,p_i^-}(\mathbb{R}^N)$. Note that a careful inspection of the proof of \cite[Theorem 6.5]{DPV} shows that we can choose the constant $C$ in Proposition~\ref{P-S.inequality} as $C=sp\ \omega_N^{\frac{N+sp}{N}}2^{p+p_s^\ast}.$ Applying this result we obtain
\begin{equation*}
\|\widetilde{v}\|^{p_i^-}_{L^{(p_i^-)_t^\ast}(\mathbb{R}^N)}\leq tp_i^-\ \omega_N^{\frac{N+tp_i^-}{N}}2^{p_i^-+(p_i^-)_t^\ast}\|\widetilde{v}\|_{t,p_i^-,\mathbb{R}^N}.
\end{equation*}
This yields
\begin{equation*}
\|v\|_{L^{r_i^+}(Q_i)}\leq C \|\widetilde{v}\|_{t,p_i^-,\mathbb{R}^N}.
\end{equation*}
Combining this with \eqref{prf.theo3.5.est1}-\eqref{prf.theo3.5.est3}
we obtain \eqref{P.T.loc.const}.

We next consider two cases.
\begin{itemize}
	\item Case $\|v\|_{L^{r(\cdot)}(Q_i)}\geq 1.$ Then, invoking Propositions~\ref{norm-modular} and \ref{norm-modular2} with taking \eqref{P.T.norm.Qi} into account, we have
\end{itemize}
\begin{align*}
\int_{Q_i}|v|^{r(x)}\diff x\leq \|v\|_{L^{r(\cdot)}(Q_i)}^{r_i^+}&\leq C^{r_i^+}|v|_{s,p,Q_i}^{r_i^+}\notag\\
&\leq C^{r_i^+}\left(\int_{Q_i}|v|^{p(x)}\diff x+
\int_{Q_i}\int_{Q_i}
\frac{|v(x)-v(y)|^{p(x,y)}}{|x-y|^{N+sp(x,y)}}
\,\diff x\diff y\right)^{\frac{r_i^+}{p_i^+}}\notag\\
&\leq C^{r_i^+}\left(\int_{Q_i}|v|^{p(x)}\diff x+
\int_{Q_i}\int_{Q_i}
\frac{|v(x)-v(y)|^{p(x,y)}}{|x-y|^{N+sp(x,y)}}
\,\diff x\diff y\right).
\end{align*}
\begin{itemize}
	\item Case $\|v\|_{L^{r(\cdot)}(Q_i)}< 1.$ Then, invoking Propositions~\ref{norm-modular} and \ref{norm-modular2} with taking \eqref{P.T.norm.Qi} into account again, we have
\end{itemize}
\begin{align*}
\int_{Q_i}|v|^{r(x)}\diff x\leq \|v\|_{L^{r(\cdot)}(Q_i)}^{r_i^-}&\leq C^{r_i^-}|v|_{s,p,Q_i}^{r_i^-}\notag\\
&\leq C^{r_i^-}\left(\int_{Q_i}|v|^{p(x)}\diff x+
\int_{Q_i}\int_{Q_i}
\frac{|v(x)-v(y)|^{p(x,y)}}{|x-y|^{N+sp(x,y)}}
\,\diff x\diff y\right)^{\frac{r_i^-}{p_i^+}}\notag\\
&\leq C^{r_i^-}\left(\int_{Q_i}|v|^{p(x)}\diff x+
\int_{Q_i}\int_{Q_i}
\frac{|v(x)-v(y)|^{p(x,y)}}{|x-y|^{N+sp(x,y)}}
\,\diff x\diff y\right).
\end{align*}
It follows that in any case, for any $i\in\mathbb{N}$ we have
\begin{equation*}
\int_{Q_i}|v|^{r(x)}\diff x\leq \left(C^{r^-}+C^{r^+}\right)\left(\int_{Q_i}|v|^{p(x)}\diff x+
\int_{Q_i}\int_{Q_i}
\frac{|v(x)-v(y)|^{p(x,y)}}{|x-y|^{N+sp(x,y)}}
\,\diff x\diff y\right).
\end{equation*}
Summing up the last inequality over all $i\in\mathbb{N}$, combining with \eqref{P.T.v}, we obtain
\begin{equation*}
\int_{\mathbb{R}^N}|v|^{r(x)}\diff x\leq C^{r^-}+C^{r^+}.
\end{equation*}
Thus, $W^{s,p(\cdot,\cdot)}(\mathbb{R}^N)\subset L^{r(\cdot)}(\mathbb{R}^N)$ and hence, $W^{s,p(\cdot,\cdot)}(\mathbb{R}^N)\hookrightarrow L^{r(\cdot)}(\mathbb{R}^N)$ due to the closed graph theorem. The proof of assertion (i) is complete.

\vspace{0.3cm} (ii) Let $B$ be  any ball in $\mathbb{R}^N$. Let  $u_n\rightharpoonup 0$ in $W^{s,p(\cdot,\cdot)}(\mathbb{R}^N)$ and
thus $u_n\rightharpoonup 0$ in $W^{s,p(\cdot,\cdot)}(B).$ Invoking Theorem~\ref{general.imb} we have $u_n\to 0$ in $L^{r(\cdot)}(B).$ The proof is complete.
	\end{proof}

\section{A-priori bounds for solutions}\label{Sec.a-prioribound}
In this section, we obtain a-priori bounds for solutions to problem
\begin{equation}\label{eq.bounds}
\begin{cases}
(-\Delta)_{p(x)}^su=f(x,u) \quad \text{in} \quad \Omega, \\
u=0 \quad \text{in} \quad \Bbb R^N\setminus \Omega,
\end{cases}
\end{equation}
where $\Omega$ is a bounded Lipschitz domain in $\Bbb R^N$; $s\in (0,1)$; $p\in C\left(\Bbb R^N\times \Bbb R^N\right)$  is symmetric such
that $1<\underset{(x,y)\in \Bbb R^N\times \Bbb R^N}{\inf}\ p(x,y)\le
\underset{(x,y)\in \Bbb R^N\times \Bbb R^N}{\sup}\ p(x,y)<\frac{N}{s};$  and the nonlinear term $f$ satisfies that
\begin{itemize}
	\item[\AssF1] $f:\Omega\times \Bbb R\to \Bbb R$ is a
	Carath\'eodory function such that
	$$
	|f(x,t)|\le C\round{1+|t|^{q(x)-1}}, \ \text{for a.e.}\ x\in \Omega\ \text{and for all}\ t\in\mathbb{R},
	$$
	for some positive constant $C$, where $q\in C_+(\ol{\Omega})$ such that $q(x)<p_s^*(x)$ for all $x\in\ol{\Omega}.$ 
	\end{itemize}
We look for solutions of problem \eqref{eq.bounds} in the space
$$W_0^{s,p(\cdot,\cdot)}(\Omega):=\left\{u\in W^{s,p(\cdot,\cdot)}(\mathbb{R}^N):\ u=0\quad \text{in}\quad \mathbb{R}^N\setminus\Omega\right\}.$$
Clearly, $W_0^{s,p(\cdot,\cdot)}(\Omega)$ is a closed subspace of $W^{s,p(\cdot,\cdot)}(\mathbb{R}^N)$ and hence, $W_0^{s,p(\cdot,\cdot)}(\Omega)$ a reflexive separable Banach space with the norm $\|\cdot\|_{s,p,\mathbb{R}^N}.$ Furthermore, $W_0^{s,p(\cdot,\cdot)}(\Omega)\hookrightarrow W^{s,p(\cdot,\cdot)}(\Omega)$ and hence, as a consequence of Corollary~\ref{special.imb} we have
\begin{equation}\label{compact.imb}
W_0^{s,p(\cdot,\cdot)}(\Omega)\hookrightarrow\hookrightarrow L^{q(\cdot)}(\Omega).
\end{equation}
This makes the following definition well-defined.
\begin{definition}\label{def.weak.sol}
We say that $u\in W_0^{s,p(\cdot,\cdot)}(\Omega)$ is a (weak) solution of problem \eqref{eq.bounds} if
\begin{equation}\label{weak.eq}
\int_{\mathbb{R}^N}\int_{\mathbb{R}^N}
\frac{|u(x)-u(y)|^{p(x,y)-2}(u(x)-u(y))(v(x)-v(y))}{|x-y|^{N+sp(x,y)}}
\,\diff x\diff y =\int_{\Omega}f(x,u)
v\,\diff x
\end{equation}
for all $v \in W_0^{s,p(\cdot,\cdot)}(\Omega)$.
\end{definition}
Our main result in this section is the following.

\begin{theorem}\label{a-priori.bound}
Let $p,s$ be as above such that
\begin{equation}\label{LH}
\inf_{\epsilon>0}\ \underset{\underset{0<|x-y|<1/2}{(x,y)\in\mathbb{R}^N\times\mathbb{R}^N}}{\sup} \left|p(x,y)-p_{B_{\epsilon}(x,y)}^-\right|\log \frac{1}{|x-y|} <\infty,
\end{equation}
where $p_{B_{\epsilon}(x,y)}^-:=\underset{(x',y')\in B_{\epsilon}(x,y)}{\inf}p(x',y').$ Then, under the assumption $\textup{(F1)}$ a weak solution $u$ to problem \eqref{eq.bounds} belongs to $L^{\infty}(\Omega)$ and there exist $C,\tau_1, \tau_2$
independent of $u$ such that
\begin{equation}\label{L^infity}
\|u\|_{L^{\infty}(\Omega)}\le
C\max\left\{\|u\|^{\tau_1}_{L^{\widetilde{q}(\cdot)}(\Omega)},\|u\|^{\tau_2}_{L^{\widetilde{q}(\cdot)}(\Omega)}\right\}
\end{equation}
for $\widetilde{q}(x):=\max\{p(x),q(x)\}.$
\end{theorem}
In the following example, we provide a non-constant exponent $p$ that fulfills the conditions in Theorem~\ref{a-priori.bound}.
\begin{example}\rm
Let $R>0$ be such that $\ol{\Omega}\times \ol{\Omega}\subset
B_R(0,0)$ and let $\xi_R\in C_c^{\infty}(\Bbb R^N\times \Bbb R^N)$ be
such that $0\leq \xi_R\leq 1,$ $\xi_R=1$ on $\ol{\Omega}\times \ol{\Omega}$ and
$\operatorname{supp}(\xi_R)\subset B_R(0,0)$. Let $p(x,y)=p_0+|x-y|\xi_R(x,y)$ on $\mathbb{R}^N\times\mathbb{R}^N$ 
for some constant $p_0>1$. Thus,
$$1<p_0\leq p(x,y)\leq p_0+\underset{(x,y)\in B_R(0,0)}{\sup}|x-y|\xi_R(x,y)<\infty$$
and for any $\epsilon>0,$
$$0\leq p(x,y)-p_{B_{\epsilon}(x,y)}^-\le p(x,y)-p_0\le |x-y|, \ \forall (x,y)\in\mathbb{R}^N\times\mathbb{R}^N.
$$
This implies that for all $\epsilon>0$ and for all $(x,y)\in\mathbb{R}^N\times\mathbb{R}^N$ with $|x-y|\le \frac{1}{2}$ we have
$$\left|p(x,y)-p_{B_{\epsilon}(x,y)}^-\right|\log\frac{1}{|x-y|}\le \underset{\underset{0<|x-y|<1/2}{(x,y)\in\mathbb{R}^N\times\mathbb{R}^N}}{\sup}\left|x-y\right|\log\frac{1}{|x-y|}<\infty$$
due to the fact that $\lim_{t\to 0^+}t\log t=0,$ and hence, \eqref{LH} holds.
\end{example}
\bigskip
To prove Theorem~\ref{a-priori.bound} we employ the De Giorgi iteration argument used in \cite{HS2014} for which the following result is essential.
\begin{lemma}{\rm (\cite[Lemma 4.3]{HS2014})}\label{leRecur}
	Let $\{Z_n\}_{n=0}^{\infty}$  be a sequence  of positive numbers satisfying the recursion inequality
	\begin{equation}\label{Recur.Int}
	Z_{n+1}\leq K b^n \left( Z_{n}^{1+\delta_1}+Z_{n}^{1+\delta_2}\right),\quad n=0,1,2,\cdots,
	\end{equation}
	for some $b>1, \ K>0 \ and\ \delta_2 \geq \delta_1 >0$. If $Z_{0}\leq \min\left(1,(2K)^{\frac{-1}{\delta_1}}\ b^{\frac{-1}{\delta_1^{2}}}\right) $ or
	\begin{equation*}\label{Y0.Le}
	Z_{0}\leq \min\left((2K)^{\frac{-1}{\delta_1}}\ b^{\frac{-1}{\delta_1^{2}}},(2K)^{\frac{-1}{\delta_2}}\ b^{-\frac{1}{\delta_1\delta_2}-\frac{\delta_2-\delta_1}{\delta_2^{2}}}\right),
	\end{equation*}
	then there exists $n \in \mathbb{N}\cup \{0\}=:\mathbb{N}_0$ such that $Z_n\leq 1$. Moreover,
	$$
	Z_{n}\leq \min\left(1,(2K)^{\frac{-1}{\delta_1}}\ b^{\frac{-1}{\delta_1^{2}}}\ b^{\frac{-n}{\delta_1}}\right),\ \forall n\geq n_0,
	$$
	where $n_0$ is the smallest $n\in \mathbb{N}_0$ for which $Z_n\leq 1$. In particular, $Z_n \to 0$ as $n \to \infty$.
\end{lemma}

\begin{proof}[Proof of Theorem~\ref{a-priori.bound}] Let $u$ be a weak solution to problem~\eqref{eq.bounds}. Note that by replacing $q$ with $\widetilde{q}$ if necessary, we may assume $p(x)\leq q(x)$ for all $x\in\overline{\Omega}.$ The proof includes several steps.   
	
	\textbf{Step 1. Constructing the recursion sequence $\{Z_n\}$ and basic estimates.}
	
	\noindent Define
$$Z_n:=\int_{A_{k_n}}(u-k_n)^{q(x)}\,\diff x,$$
where $k_n:=k_*(2-\frac{1}{2^n}),
 \ n\in \mathbb{N}_0$ for $k_*>0$ to be specified later and $A_k:=\{x\in \Omega: u(x)>k\}$ for $k>0.$  Clearly, $k_n \uparrow 2k_{\ast},$ $k_{\ast} \leq k_n < 2k_{\ast},$ and $k_{n+1}-k_n = \frac{k_{\ast}}{2^{n+1}}$ for all $n \in \mathbb{N}_0.$
 We have the following estimates:
 \begin{equation}\label{N001}
 \int_{A_{k_{n+1}}}u^{q(x)}\,\diff x\le 2^{(n+2)q^+}Z_n,\quad \forall n\in\mathbb{N}_0;
 \end{equation}
 \begin{equation}\label{N002}
 |A_{k_{n+1}}|\le \round{k_*^{-q^-}+k_*^{-q^+}}2^{(n+1)q^+}Z_n\le 2\round{1+k_*^{-q^+}}2^{(n+1)q^+}Z_n,\quad \forall n\in\mathbb{N}_0;
 \end{equation}
 and
 \begin{equation}\label{grad}
 \int_{\Bbb R^N}\int_{\Bbb R^N}
 \frac{|u_n(x)-u_n(y)|^{p(x,y)}}{|x-y|^{N+sp(x,y)}} \,\diff x\diff y \le
 C_1\left(1+k_*^{-q^+}\right)2^{nq^+}Z_n,\quad \forall n\in\mathbb{N}_0,
 \end{equation}
 where $u_n:=(u-k_{n+1})_+.$ Here and in the rest of the proof, $C_i$ ($i\in\mathbb{N}$) is a positive  constant independent of $u$ and $n$ and $v_+:=\max\{v,0\}.$

 Indeed, we have
\begin{align*}
Z_n&=\int_{A_{k_n}}(u-k_n)^{q(x)}\,\diff x\\
&\ge\int_{A_{k_{n+1}}}\round{u-\frac{u}{k_{n+1}}k_n}^{q(x)}\,\diff x =\int_{A_{k_{n+1}}}\round{1-\frac{k_n}{k_{n+1}}}^{q(x)}u^{q(x)}\,\diff x\\
&\ge\int_{A_{k_{n+1}}}\frac{u^{q(x)}}{(2^{n+2}-1)^{q(x)}}\,\diff x\ge
\int_{A_{k_{n+1}}}\frac{u^{q(x)}}{2^{(n+2)q^+}}\,\diff x,
\end{align*}
and hence, \eqref{N001} follows. The estimate \eqref{N002} for the measures of the level sets $A_{k_{n}}$ follows from
$$|A_{k_{n+1}}|\le
\int_{A_{k_{n+1}}}\round{\frac{u-k_n}{k_{n+1}-k_n}}^{q(x)}\,\diff x\le
\int_{A_{k_{n}}}\round{\frac{2^{n+1}}{k_*}}^{q(x)}(u-k_n)^{q(x)}\,\diff x.
$$
Finally, to obtain \eqref{grad} we first note that $(u-k)_+\in W_0^{s,p(\cdot,\cdot)}(\Omega)$ for all $u\in W_0^{s,p(\cdot,\cdot)}(\Omega)$ and all $k\geq 0.$ Using $u_n$ as a test function in \eqref{weak.eq}, we obtain
\begin{align*}
\int_{\Bbb R^N}\int_{\Bbb R^N}&
\frac{|u(x)-u(y)|^{p(x,y)-2}(u(x)-u(y))(u_n(x)-u_n(y))}{|x-y|^{N+sp(x,y)}}
\,\diff x \diff y=\int_{\Omega}f(x,u)u_n(x)\,\diff x.\\
\end{align*}
 Using the inequality $|\alpha-\beta|^{\gamma-2}(\alpha-\beta)(\alpha_+-\beta_+)\geq |\alpha_+-\beta_+|^\gamma$ for all $\alpha,\beta,\gamma\in\mathbb{R}$ with $\gamma>1$ and $\textup{(F1})$, we deduce from the last equality that
\begin{align*}
 \int_{\Bbb R^N}\int_{\Bbb R^N}
\frac{|u_n(x)-u_n(y)|^{p(x,y)}}{|x-y|^{N+sp(x,y)}} \,\diff x \diff y \le C\int_{\Omega}\left[1+|u|^{q(x)-1}\right]u_n(x)\,\diff x.
\end{align*}
Thus
\begin{align*}
 \int_{\Bbb R^N}\int_{\Bbb R^N}
\frac{|u_n(x)-u_n(y)|^{p(x,y)}}{|x-y|^{N+sp(x,y)}} \,\diff x \diff y &\le
C\int_{A_{k_{n+1}}}u^{q(x)}\,\diff x+C\int_{A_{k_{n+1}}}u\,\diff x
\\
&\le C\int_{A_{k_{n+1}}}u^{q(x)}\,\diff x+C\int_{A_{k_{n+1}}}u^{1-q(x)}u^{q(x)}\,\diff x\\
&\le C\int_{A_{k_{n+1}}}u^{q(x)}\,\diff x+C\left(1+k_*^{1-q^+}\right)\int_{A_{k_{n+1}}}u^{q(x)}\,\diff x,
\end{align*}
and hence, \eqref{grad} follows.

\bigskip

\textbf{Step 2. Localization.}

\noindent Let $\{B_{i}\}_{i=1}^{m}$ be a covering of $\ol{\Omega},$ where $B_{i}:=B_\epsilon(z_i)$ with $\epsilon>0$ sufficiently small such that $B_i\cap\Omega$ ($i\in\{1,\cdots,m\}$) are Lipschitz domains as well. By \eqref{LH} and the continuity of $p,q$, we can choose $\e\in (0,1/4)$ sufficiently small such that
\begin{equation}\label{LH.loc}
-\left|p(x,y)-p^-_{B_{4\epsilon}(x,y)}\right|\log |x-y|\leq C_2,\quad \forall x,y\in\mathbb{R}^N,\ |x-y|<\frac{1}{2},
\end{equation}
and
\begin{equation}\label{localized. exponents}
p_i^-:=\inf_{(x,y)\in B_i\times B_i}p(x,y)\le q_{i}^-:=\inf_{x\in B_i\cap \Omega} q(x)\le q_{i}^+:=\sup_{x\in B_i\cap \Omega} q(x)<(p_i^-)_s^*.
\end{equation}
Let $\{\xi_{i}\}_{i=1}^{m}$ be a partition of unity of $\ol{\Omega}$ associated with
$\{B_i\}_{i=1}^{m}$, that is, for each $i\in\{1,\cdots,m\},$ $\xi_{i}\in
C_c^{\infty}(\Bbb R^N),\ \operatorname{supp}(\xi_{i})\subset
B_i,$ $0\le \xi_{i}\le 1,$ and
\begin{equation}\label{xi.i}
\sum_{i=1}^{m}\xi_{i}=1 \ \text{on} \ \ol{\Omega}.
\end{equation}
The following assertion holds:
\begin{equation}\label{localization}
\int_{B_i}\int_{B_i}\frac{|u_n(x)-u_n(y)|^{p_i^-}}{|x-y|^{N+sp_i^-}} \,\diff x \diff y\le C_3\round{
	\int_{B_i}\int_{B_i}
	\frac{|u_n(x)-u_n(y)|^{p(x,y)}}{|x-y|^{N+sp(x,y)}}
	\,\diff x \diff y+|A_{k_{n+1}}|},
\end{equation}
and hence,
\begin{equation}\label{localization1}
\int_{B_i}\int_{B_i}
\frac{|u_n(x)-u_n(y)|^{p_i^-}}{|x-y|^{N+sp_i^-}} \,\diff x \diff y\le C_4\round{1+k_*^{-q^+}}2^{nq^+}Z_n,
\end{equation}
for all $i\in\{1,2,\cdots , m\}$ and  all $n\in \Bbb N.$

\noindent To prove \eqref{localization}, let $R>1$ be such that $\Omega-\Omega\subset B_{R-1}(0).$ 
For each $i\in\{1,\cdots,m\}$, we have
\begin{equation}\label{I}
\int_{B_i} \frac{\diff x}{|x-y|^{N-sp_i^-}}\le\int_{B_R(0)}
\frac{\diff z}{|z|^{N-sp_i^-}}= \frac{\omega_NR^{sp_i^-}}{sp_i^-}\leq \frac{\omega_NR^{sp^+}}{sp^-},\ \forall y\in\Omega.
\end{equation}
Fix $i\in\{1,\cdots,m\}$ and $n\in\mathbb{N}.$ We have
\begin{align*}
\int_{B_i}\int_{B_i} \frac{|u_n(x)-u_n(y)|^{p(x,y)}}{|x-y|^{N+sp(x,y)}}
\,\diff x \diff y&=\int_{B_i\cap A_{k_{n+1}}}\int_{B_i\cap
A_{k_{n+1}}} \frac{|u_n(x)-u_n(y)|^{p(x,y)}}{|x-y|^{N+sp(x,y)}}
\,\diff x \diff y\\
&\quad +2\int_{B_i\cap
A_{k_{n+1}}}\int_{B_i\setminus A_{k_{n+1}}}
\frac{|u_n(x)-u_n(y)|^{p(x,y)}}{|x-y|^{N+sp(x,y)}} \,\diff x \diff y \\
&=: I_1+2I_2.
\end{align*}
To estimate $I_1$ and $I_2$, we note that \eqref{LH.loc} implies that
\begin{equation}\label{LH'}
|x-y|^{s(p(x,y)-p_i^-)}=e^{s(p(x,y)-p_i^-)\log{|x-y|}}\ge C_5,\quad \forall (x,y)\in B_i\times B_i.
\end{equation}
Ineed, let $p_i^-=p(x_i,y_i)$ for some $(x_i,y_i)\in \bar{B_i}\times \bar{B_i}.$ Clearly, $|(x_i,y_i)-(x,y)|=|x-x_i|+|y-y_i|<4\epsilon$, and hence, $(x_i,y_i)\in B_{4\epsilon}(x,y)$ for all $(x,y)\in B_i\times B_i.$ Moreover, $|x-y|<2\epsilon<\frac{1}{2}$  for all $(x,y)\in B_i\times B_i.$ Thus,
$$-(p(x,y)-p_i^-)\log{|x-y|}\leq -\left(p(x,y)-p^-_{B_{4\epsilon}(x,y)}\right)\log |x-y|\leq C_2,\quad \forall (x,y)\in B_i\times B_i,$$
and hence, we obtain \eqref{LH'}.

\noindent Denote $\Sigma_{i,n}=B_i\cap A_{k_{n+1}}$ and $\mathcal{C}_{i,n}=B_i\setminus A_{k_{n+1}}$ for brevity. Using \eqref{LH'}, we estimate $I_1$ as follows:
\begin{align}\label{I1.1}
I_1&=\int_{\Sigma_{i,n}}\int_{\Sigma_{i,n}}
\abs{\frac{|u_n(x)-u_n(y)|}{|x-y|^{2s}}}^{p(x,y)}\cdot
\frac{1}{|x-y|^{N-sp_i^-}} \cdot \frac{1}{|x-y|^{-s(p(x,y)-p_i^-)}}
\,\diff x \diff y\notag\\
&\ge C_5\int_{\Sigma_{i,n}}\int_{\Sigma_{i,n}}
\round{\abs{\frac{|u_n(x)-u_n(y)|}{|x-y|^{2s}}}^{p_i^-}-1}\cdot
\frac{\diff x \diff y}{|x-y|^{N-sp_i^-}}\notag\\
&\ge C_5\int_{\Sigma_{i,n}}\int_{\Sigma_{i,n}}
\frac{|u_n(x)-u_n(y)|^{p_i^-}}{|x-y|^{N+sp_i^-}}\,\diff x \diff y
-C_5\int_{A_{k_{n+1}}}\round{\int_{B_i}
\frac{\diff x}{|x-y|^{N-sp_i^-}}}\diff y.
\end{align}
From this and \eqref{I}, we obtain
\begin{equation}\label{I1}
I_1\ge C_5\int_{B_i\cap A_{k_{n+1}}}\int_{B_i\cap A_{k_{n+1}}}
\frac{|u_n(x)-u_n(y)|^{p_i^-}}{|x-y|^{N+sp_i^-}}\,\diff x \diff y-C_6|A_{k_{n+1}}|.
\end{equation}
On the other hand, using \eqref{LH'} again we have
\begin{align*}
I_2&=\int_{\Sigma_{i,n}}\int_{\mathcal{C}_{i,n}}
\abs{\frac{|u_n(x)-u_n(y)|}{|x-y|^{2s}}}^{p(x,y)}\frac{1}{|x-y|^{-s(p(x,y)-p_i^-)}}
\frac{\diff x \diff y}{|x-y|^{N-sp_i^-}}\\
&\ge C_5\int_{\Sigma_{i,n}}\int_{\mathcal{C}_{i,n}}
\round{\abs{\frac{|u_n(x)-u_n(y)|}{|x-y|^{2s}}}^{p_i^-}-1}
\frac{\diff x \diff y}{|x-y|^{N-sp_i^-}}\\
&\ge C_5\int_{\Sigma_{i,n}}\int_{\mathcal{C}_{i,n}}
\frac{|u_n(x)-u_n(y)|^{p_i^-}}{|x-y|^{N+sp_i^-}}\,\diff x \diff y-C_5
\int_{A_{k_{n+1}}}\round{\int_{B_i}\frac{\diff x}{|x-y|^{N-sp_i^-}}}\diff y.
\end{align*}
This and \eqref{I} yield
$$I_2\ge C_5\int_{B_i\cap A_{k_{n+1}}}\int_{B_i\setminus A_{k_{n+1}}}
\frac{|u_n(x)-u_n(y)|^{p_i^-}}{|x-y|^{N+sp_i^-}}\,\diff x \diff y-C_6
|A_{k_{n+1}}|.$$
Combining this with \eqref{I1} we obtain
$$
\int_{B_i}\int_{B_i} \frac{|u_n(x)-u_n(y)|^{p(x,y)}}{|x-y|^{N+sp(x,y)}}
\,\diff x \diff y\ge C_5\int_{B_i}\int_{B_i}
\frac{|u_n(x)-u_n(y)|^{p_i^-}}{|x-y|^{N+sp_i^-}}
\,\diff x \diff y-3C_6|A_{k_{n+1}}|.
$$
Thus, \eqref{localization} has been proved. The inequality \eqref{localization1} then follows from \eqref{N002}, \eqref{grad} and \eqref{localization}.

\medskip
\textbf{Step 3. Estimate $Z_{n+1}$ by $Z_n$.}

\noindent From \eqref{xi.i} we have $u_n=\sum_{i=1}^{m}\xi_{i}u_n$ on $\Omega$ and hence, using Jensen's inequality we have
\begin{equation}\label{decomposition.Zn+1}
Z_{n+1}=\int_{A_{k_{n+1}}}u_n^{q(x)}\,\diff x\le
m^{q^+-1}\sum_{i=1}^m\round{\int_{A_{k_{n+1}}}|u_n\xi_i|^{q_i^+}\,\diff x+\int_{A_{k_{n+1}}}|u_n\xi_i|^{q_i^-}\,\diff x}.
\end{equation}
For each $i\in \{1,\cdots,m\}$, fix $\widetilde{q}_i\in \left(q_i^+, (p_i^-)_s^*\right)$ and let $\ol{q}_i\in
\{q_i^+,q_i^-\}$. Using H\"older's inequality we have
\begin{equation}\label{*1}
\int_{A_{k_{n+1}}}|u_n\xi_i|^{\bar{q}_i}\, \diff x\le\round{\int_{\Omega}|u_n\xi_i|^{\widetilde{q}_i}\,\diff x}^{\frac{\ol{q}_i}{\widetilde{q}_i}}|A_{k_{n+1}}|^{1-\frac{\ol{q}_i}{\widetilde{q}_i}}.
\end{equation}
Using H\"older's inequality and then invoking Proposition~\ref{P-S.inequality}, we have
\begin{align}\label{*2}
\nonumber\|u_n\xi_i\|_{L^{\widetilde{q}_i}(\Omega)}\leq C_7\|u_n\xi_i\|_{L^{(p_i^-)_s^\ast}(\Omega)}&
\le  C_8\round{ \int_{\Bbb R^N}\int_{\Bbb R^N}
\frac{|u_n(x)\xi_i(x)-u_n(y)\xi_i(y)|^{p_i^-}}{|x-y|^{N+sp_i^-}}
\,\diff x \diff y}^{\frac{1}{p_i^-}}.
\end{align}
Hence,
\begin{equation}\label{un.loc}
\|u_n\xi_i\|^{p_i^-}_{L^{\widetilde{q}_i}(\Omega)}\leq  C_9\round{ \int_{\Bbb R^N}\int_{\Bbb R^N} \frac{|u_n(x)\xi_i(x)-u_n(y)\xi_i(y)|^{p_i^-}}{|x-y|^{N+sp_i^-}}
	\,\diff x \diff y}=C_9\round{J_1+2J_2},
\end{equation}
where
$$J_1:=\int_{B_i}\int_{B_i}
\frac{|u_n(x)\xi_i(x)-u_n(y)\xi_i(y)|^{p_i^-}}{|x-y|^{N+sp_i^-}}
\,\diff x \diff y$$
and
$$J_2:=\int_{\mathbb{R}^N\setminus B_i}\int_{B_i}
\frac{|u_n(x)\xi_i(x)-u_n(y)\xi_i(y)|^{p_i^-}}{|x-y|^{N+sp_i^-}}
\,\diff x \diff y.$$
We have
\begin{align*}
J_1&=\int_{B_i}\int_{B_i}\frac{|(u_n(x)-u_n(y))\xi_i(x)+u_n(y)(\xi_i(x)-\xi_i(y))|^{p_i^-}}{|x-y|^{N+sp_i^-}}
\,\diff x \diff y\\
&\le 2^{p_i^-}\int_{B_i}\int_{B_i}
\frac{|u_n(x)-u_n(y)|^{p_i^-}}{|x-y|^{N+sp_i^-}} \,\diff x \diff y\\
&\hspace{3cm}+
2^{p_i^-}\|\nabla
\xi_i\|_{\infty}^{p_i^-}\int_{B_i}\round{\int_{B_i}
\frac{\diff x}{|x-y|^{N+(s-1)p_i^-}}} |u_n(y)|^{p_i^-} \,\diff y.
\end{align*}
Combining this with the estimate $$\int_{B_i}
\frac{\diff x}{|x-y|^{N+(s-1)p_i^-}}\leq\int_{B_R(0)}
\frac{\diff z}{|z|^{N+(s-1)p_i^-}}=\frac{\omega_NR^{(1-s)p_i^-}}{(1-s)p_i^-}\leq \frac{\omega_NR^{(1-s)p^+}}{(1-s)p^-},$$
we infer
\begin{align}\label{est.J1}
J_1&\le 2^{p_i^-}\int_{B_i}\int_{B_i}
\frac{|u_n(x)-u_n(y)|^{p_i^-}}{|x-y|^{N+sp_i^-}} \,\diff x \diff y+
C_{10}\int_{B_i}|u_n(y)|^{p_i^-} \,\diff y\notag\\
&\le 2^{p_i^-}\int_{B_i}\int_{B_i}
\frac{|u_n(x)-u_n(y)|^{p_i^-}}{|x-y|^{N+sp_i^-}} \,\diff x \diff y+
C_{10}\int_{A_{k_{n+1}}\cap B_i}u(y)^{p_i^-} \,\diff y.
\end{align}
On the other hand, we have
\begin{align}
J_2&=\int_{\mathbb{R}^N\setminus B_i}\round{\int_{B_i}
	\frac{|u_n(x)\xi_i(x)|^{p_i^-}}{|x-y|^{N+sp_i^-}}\diff x} \,\diff y\notag\\
&\leq\int_{\operatorname{supp}(\xi_i)\cap A_{k_{n+1}}}\round{\int_{\mathbb{R}^N\setminus B_i}
	\frac{\diff y}{|x-y|^{N+sp_i^-}}} |u_n(x)|^{p_i^-} \,\diff x.
\end{align}
From this and the estimate $\underset{x\in \operatorname{supp}(\xi_i)}{\sup}\int_{\mathbb{R}^N\setminus B_i}
\frac{\diff y}{|x-y|^{N+sp_i^-}}\leq \int_{|z|\geq d_i}\frac{\diff z}{|z|^{N+sp_i^-}}=\frac{\omega_N}{sp_i^-d_i^{sp_i^-}},$ where $d_i:=\operatorname{dist}(\mathbb{R}^N\setminus B_i,\operatorname{supp}(\xi_i))>0,$ we obtain
\begin{equation*}
J_2\leq C_{11}\int_{A_{k_{n+1}}\cap B_i}u(x)^{p_i^-} \,\diff x.
\end{equation*}
Combining this with \eqref{un.loc} and \eqref{est.J1}, we derive
\begin{align*}
\|u_n\xi_i\|^{p_i^-}_{L^{\widetilde{q}_i}(\Omega)}&\le C_{12}\int_{B_i}\int_{B_i}
\frac{|u_n(x)-u_n(y)|^{p_i^-}}{|x-y|^{N+sp_i^-}} \,\diff x \diff y+
C_{12}\int_{A_{k_{n+1}}\cap B_i}u(x)^{p_i^-} \,\diff x\notag\\
&\le C_{12}\int_{B_i}\int_{B_i}
\frac{|u_n(x)-u_n(y)|^{p_i^-}}{|x-y|^{N+sp_i^-}} \,\diff x \diff y+
C_{12}\left(1+k_*^{p_i^--q_i^+}\right)\int_{A_{k_{n+1}}}u^{q(x)} \,\diff x.
\end{align*}
From this and \eqref{N001} and \eqref{localization1}, noting $1+k_*^{p_i^--q_i^+}\leq 2\round{1+k_*^{-q^+}},$ we obtain
\begin{equation*}
\|u_n\xi_i\|^{p_i^-}_{L^{\widetilde{q}_i}(\Omega)}\le C_{13}\round{1+k_*^{-q^+}}2^{nq^+}Z_n.
\end{equation*}
Combining this with \eqref{*1} and \eqref{N002}, we have
\begin{align*}
\int_{\Omega} |u_n\xi_i|^{\ol{q}_i} \,\diff x\leq C_{14}\round{1+k_\ast^{-q^+}}^{\frac{\bar{q}_i}{p_i^-}}2^{\frac{n\bar{q}_iq^+}{p_i^-}}Z_n^{\frac{\bar{q}_i}{p_i^-}}\round{k_\ast^{-q^-}+k_\ast^{-q^+}}^{1-\frac{\bar{q}_i}{\widetilde{q}_i}}2^{nq^+\round{1-\frac{\bar{q}_i}{\widetilde{q}_i}}}Z_n^{1-\frac{\bar{q}_i}{\widetilde{q}_i}}.
\end{align*}
Hence,
\begin{equation*}
\int_{\Omega} |u_n\xi_i|^{\ol{q}_i} \,\diff x\leq C_{15}\round{k_\ast^{-q^-\round{1-\frac{\bar{q}_i}{\widetilde{q}_i}}}+k_\ast^{-q^+\round{1-\frac{\bar{q}_i}{\widetilde{q}_i}+\frac{\bar{q}_i}{p_i^-}}}}2^{nq^+\round{1-\frac{\bar{q}_i}{\widetilde{q}_i}+\frac{\bar{q}_i}{p_i^-}}}Z_n^{1-\frac{\bar{q}_i}{\widetilde{q}_i}+\frac{\bar{q}_i}{p_i^-}}.
\end{equation*}
This yields
\begin{equation}\label{Zn+1.loc}
\int_{\Omega} |u_n\xi_i|^{\ol{q}_i} \,\diff x\leq C_{15}(k_*^{-\gamma_1}+k_*^{-\gamma_2})b^n\round{Z_{n}^{1+\delta_1}+Z_{n}^{1+\delta_2}},
\end{equation}
where
$$0<\gamma_1:=\min_{1\leq i\leq m}\ q^-\round{1-\frac{q_i^+}{p_i^-}}<\gamma_2:=\max_{1\leq i\leq m}\ q^+\round{1-\frac{q_i^+}{\widetilde{q}_i}+\frac{q_i^+}{p_i^-}},$$
$$b:=\max_{1\leq i\leq m}\  2^{q^+\round{1-\frac{q_i^+}{\widetilde{q}_i}+\frac{q_i^+}{p_i^-}}}>1,$$
and
$$0<\delta_1:=\min_{1\leq i\leq m}\ \round{\frac{q_i^-}{p_i^-}-\frac{q_i^-}{\widetilde{q}_i}}\leq \delta_2:=\max_{1\leq i\leq m}\ \round{\frac{q_i^+}{p_i^-}-\frac{q_i^+}{\widetilde{q}_i}} .$$
From \eqref{decomposition.Zn+1} and \eqref{Zn+1.loc}, we obtain
$$ Z_{n+1}\le
C_{16}\round{k_*^{-\gamma_1}+k_*^{-\gamma_2}}b^n\round{Z_{n}^{1+\delta_1}+Z_{n}^{1+\delta_2}},\quad \forall n\in\mathbb{N}_0.
$$

\medskip
\textbf{Step 4. A-priori bounds.}

\noindent Invoking Lemma~\ref{leRecur}, we have that
\begin{equation}\label{apply.le.1}
Z_n\to 0\quad \text{as}\quad n\to \infty
\end{equation}
provided
\begin{equation}\label{apply.le.2}
Z_0\le
\min\left\{(2C_{16}(k_*^{-\gamma_1}+k_*^{-\gamma_2}))^{-\frac{1}{\delta_1}}b^{-\frac{1}{\delta_1^2}},
\round{2C_{16}\round{k_*^{-\gamma_1}+k_*^{-\gamma_2}}}^{-\frac{1}{\delta_2}}b^{-\frac{1}{\delta_1\delta_2}-\frac{\delta_2-\delta_1}{\delta_2^2}}\right\}.
\end{equation}
To specify a $k_*$ satisfying \eqref{apply.le.2}, we first note that
\begin{equation}\label{apply.le.2'}
Z_0= \int_{\Omega}(u-k_*)_+^{q(x)}\,\diff x\le
\int_{\Omega}|u|^{q(x)}\,\diff x.
\end{equation}
On the other hand,
$$
\begin{cases}
\int_{\Omega}|u|^{q(x)}\,\diff x\le (2C_{16})^{-\frac{1}{\delta_1}}(k_*^{-\gamma_1}+k_*^{-\gamma_2})^{-\frac{1}{\delta_1}}b^{-\frac{1}{\delta_1^2}},\\
\int_{\Omega}|u|^{q(x)}\,\diff x\le
(2C_{16})^{-\frac{1}{\delta_2}}(k_*^{-\gamma_1}+k_*^{-\gamma_2})^{-\frac{1}{\delta_2}}b^{-\frac{1}{\delta_1\delta_2}-\frac{\delta_2-\delta_1}{\delta_2^2}}
\end{cases}
$$
is equivalent to
\begin{equation}\label{apply.le.3}
\begin{cases}
k_*^{-\gamma_1}+k_*^{-\gamma_2}\le (2C_{16})^{-1}b^{-\frac{1}{\delta_1}}\round{\int_{\Omega}|u|^{q(x)}\,\diff x}^{-\delta_1},\\
k_*^{-\gamma_1}+k_*^{-\gamma_2}\le
(2C_{16})^{-1}b^{-\frac{1}{\delta_1}-\frac{\delta_2-\delta_1}{\delta_2}}\round{\int_{\Omega}|u|^{q(x)}\,\diff x}^{-\delta_2}.
\end{cases}
\end{equation}
Also,
$$
\begin{cases}
2k_*^{-\gamma_1} \le (2C_{16})^{-1}b^{-\frac{1}{\delta_1}-\frac{\delta_2-\delta_1}{\delta_2}}\min\left\{\round{\int_{\Omega}|u|^{q(x)}\,\diff x}^{-\delta_1}, \round{\int_{\Omega}|u|^{q(x)}\,\diff x}^{-\delta_2}\right\},\\
2k_*^{-\gamma_2} \le
(2C_{16})^{-1}b^{-\frac{1}{\delta_1}-\frac{\delta_2-\delta_1}{\delta_2}}\min\left\{\round{\int_{\Omega}|u|^{q(x)}\,\diff x}^{-\delta_1},
\round{\int_{\Omega}|u|^{q(x)}\,\diff x}^{-\delta_2}\right\}
\end{cases}
$$
is equivalent to
\begin{equation}\label{apply.le.4}
\begin{cases}
k_*\ge (4C_{16})^{\frac{1}{\gamma_1}}b^{\frac{1}{\gamma_1}\round{\frac{1}{\delta_1}+\frac{\delta_2-\delta_1}{\delta_2}}}\max\left\{\round{\int_{\Omega}|u|^{q(x)}\,\diff x}^{\frac{\delta_1}{\gamma_1}}, \round{\int_{\Omega}|u|^{q(x)}\,\diff x}^{\frac{\delta_2}{\gamma_1}}\right\},\\
k_*\ge
(4C_{16})^{\frac{1}{\gamma_2}}b^{\frac{1}{\gamma_2}\round{\frac{1}{\delta_1}+\frac{\delta_2-\delta_1}{\delta_2}}}\max\left\{\round{\int_{\Omega}|u|^{q(x)}\,\diff x}^{\frac{\delta_1}{\gamma_2}},
\round{\int_{\Omega}|u|^{q(x)}\,\diff x}^{\frac{\delta_2}{\gamma_2}}\right\}.
\end{cases}
\end{equation}
Thus, by choosing
$$k_*=\max\left\{(4C_{16})^{\frac{1}{\gamma_1}},(4C_{16})^{\frac{1}{\gamma_2}}\right\}b^{\frac{1}{\gamma_1}\round{\frac{1}{\delta_1}+\frac{\delta_2-\delta_1}{\delta_2}}}\max\left\{\round{\int_{\Omega}|u|^{q(x)}\,\diff x}^{\frac{\delta_1}{\gamma_2}},
\round{\int_{\Omega}|u|^{q(x)}\,\diff x}^{\frac{\delta_2}{\gamma_1}}\right\},$$
we have \eqref{apply.le.4} and hence, \eqref{apply.le.3} follows. Combining this and \eqref{apply.le.2'} we get \eqref{apply.le.2} and hence, \eqref{apply.le.1} holds. On the other hand, by the Lebesgue dominated convergence theorem we have $Z_n=\int_{\Omega}(u-k_n)_+^{q(x)}\,\diff x\to \int_{\Omega}(u-2k_*)_+^{q(x)}\,\diff x$ as $n\to \infty.$ Thus, we arrive at
$$\int_{\Omega}(u-2k_*)_+^{q(x)}\,dx=0,\quad \text{and hence,}\quad \underset{x\in\Omega}{\operatorname{ess}\ \sup}\ u(x)\le 2k_*.$$
By replacing $u$ with $-u$ in the above arguments, we aso obtain
$$\underset{x\in\Omega}{\operatorname{ess}\ \inf}\ (-u)(x)\le 2k_*.$$
Therefore
\begin{equation}\label{applying.le.5}
\|u\|_{L^{\infty}(\Omega)}\le
C\max\left\{\round{\int_{\Omega}|u|^{q(x)}\,\diff x}^{\frac{\delta_1}{\gamma_2}},
\round{\int_{\Omega}|u|^{q(x)}\,\diff x}^{\frac{\delta_2}{\gamma_1}}\right\},
\end{equation}
where $C$ is a positive constant independent of $u$. Finally, we notice the following relation between the norm and the modular on $L^{q(\cdot)}(\Omega)$:
\begin{equation*}
\int_{\Omega}|u|^{q(x)}\,\diff x\leq \max\left\{\|u\|_{L^{q(\cdot)}(\Omega)}^{q^+},\|u\|_{L^{q(\cdot)}(\Omega)}^{q^-}\right\}
\end{equation*}
in view of Proposition~\ref{norm-modular}. Combining this and \eqref{applying.le.5}, we derive \eqref{L^infity} and the proof is complete.
\end{proof}

\section{Application to the fractional $p(\cdot)$-Laplacian}

In this section, we apply a modified functional method used in \cite{TF, W0} taking into account a-priori bounds for solutions obtained in the previous section to obtain the existence of infinitely many small solutions for the following problem
\begin{equation}\label{eq.existence}
	\begin{cases}
		(-\Delta)_{p(x)}^su+|u|^{p(x)-2}u=f(x,u) \quad \text{in} \quad \Omega, \\
		u=0 \quad \text{in} \quad \Bbb R^N\setminus \Omega,
	\end{cases}
\end{equation}
where $\Omega$ is a bounded Lipschitz domain in $\Bbb R^N$ and $s,p,f$ are as in Section~\ref{Sec.a-prioribound}. Furthermore, we assume in addition that:
\begin{enumerate}
\item[\AssF2] There exists a constant $t_{0}>0$ such that for a.e. $x\in\Omega$ and all $t$ with $0<\abs{t}< t_{0}$, $f$ is odd in $t$ and $p^-F(x,t) - f(x,t)t > 0,$ where $F(x,t):=\int_{0}^{t}f(x,\tau)\,\diff \tau.$
\item[\AssF3] $\lim_{t\to 0}{\frac{f(x,t)}{\abs{t}^{p^{-}-2}t}}=+\infty$ uniformly for a.e. $x\in\Omega$.
\end{enumerate}
The next theorem is our main result in this section.

\begin{theorem}\label{theorem2}
Let $s,p$ be as in Theorem \ref{a-priori.bound} and let \AssF1--\AssF3 hold. Then problem \eqref{eq.existence} has a sequence
of weak solutions $\{u_{n}\}$ satisfying
$\norm{u_{n}}_{L^{\infty}(\Omega)}\to 0$ as $n\to\infty$.
\end{theorem}

We shall use a variational argument to determine weak solutions of problem \eqref{eq.existence}. Define the functional $\Phi: W_0^{s,p(\cdot,\cdot)}(\Omega) \to \Bbb R$ as
\begin{equation}\label{Phi}
\Phi(u)=\int_{\mathbb{R}^N}\int_{\mathbb{R}^N}
\frac{|u(x)-u(y)|^{p(x,y)}}{p(x,y)|x-y|^{N+sp(x,y)}}
\,\diff x \diff y+\int_{\Omega}\frac{1}{p(x)}|u|^{p(x)}\diff x,\quad u\in W_0^{s,p(\cdot,\cdot)}(\Omega).
\end{equation}
By a standard argument, invoking the imbedding \eqref{compact.imb}, we can show that  $\Phi \in C^{1}\round{W_0^{s,p(\cdot,\cdot)}(\Omega),\Bbb R}$ and its Fr\'echet derivative $\Phi^{\prime}: W_0^{s,p(\cdot,\cdot)}(\Omega)\to \left[W_0^{s,p(\cdot,\cdot)}(\Omega)\right]^\ast$ is given by
\begin{align}\label{e:Phidef}
\langle {\Phi}^{\prime}(u),v\rangle= \int_{\mathbb{R}^N}\int_{\mathbb{R}^N}
&\frac{|u(x)-u(y)|^{p(x,y)-2}(u(x)-u(y))(v(x)-v(y))}{|x-y|^{N+sp(x,y)}}
\,\diff x \diff y\notag\\
&\hspace{3cm}+\int_{\Omega}|u|^{p(x)-2}uv\diff x,\quad \forall u,v\in W_0^{s,p(\cdot,\cdot)}(\Omega).
\end{align}
Here $\left[W_0^{s,p(\cdot,\cdot)}(\Omega)\right]^\ast$ and $\langle \cdot,\cdot \rangle$ denote the dual space of $W_0^{s,p}(\Omega)$ and the duality pairing between $W_0^{s,p}(\Omega)$ and $\left[W_0^{s,p(\cdot,\cdot)}(\Omega)\right]^\ast,$ respectively. The following $(S_+)$ property of $\Phi^{\prime}$ is useful to show the energy functionals satisfy the Palais-Smale condition (the $(\textup{PS})$ condition, for short).

\begin{lemma}\label{s+}{\rm (\cite{Bahrouni.DCDS2018})}
	The operator $\Phi^{\prime}$ is a mapping of type $(S_+)$, i.e., if
	$u_{n}\rightharpoonup u$ in $W_0^{s,p(\cdot,\cdot)}(\Omega)$ and $\lim\sup_{n\to
		\infty}\scal{\Phi^\prime(u_{n}), u_{n}-u}\le 0$, then
	$u_{n}\to u$ in $W_0^{s,p(\cdot,\cdot)}(\Omega)$ as $n\to\infty$.
\end{lemma}

\noindent The following abstract result is essential in our argument.

\begin{lemma}\label{lemma_last}{\rm (\cite{H,W0})}
	Let $X$ be a Banach space. Let $I\in C^{1}(X,{\Bbb R})$ such that $I$ satisfies the $(PS)$ condition, is even and bounded from below, and $I(0)=0$. If for any $n\in{\Bbb N}$, there
	exist an $n$-dimensional subspace $X_{n}$ and $r_{n}>0$ such
	that
	$$
	\sup_{X_{n}\cap S_{r_{n}}}{I}<0,
	$$
	where $S_{r}:=\left\{ u \in X : \norm{u}_{X}=r\right\}$, then
	$I$ has a sequence of critical values $c_{n}<0$ satisfying $c_{n}\to
	0$ as $n \to \infty$.
\end{lemma}

\begin{proof}[Proof of Theorem~\ref{theorem2}] In order to get the desired properties of the energy functional as in Lemma~\ref{lemma_last}, we modify the nonlinear term $f$ as follows. By ($\textup{F}2$) and ($\textup{F}3$), we find $t_1\in (0,t_0)$ such that
\begin{equation}\label{F1.near.0}
F(x,t)\geq |t|^{p^-}\quad \text{for a.e.}\ \ x\in\Omega\ \ \text{and all}\ \ |t|<t_1.
\end{equation}
Fix $t_{2}\in \round{0,\, t_1/2}$ and let $\varrho \in C^{1}({\Bbb R}, {\Bbb
R})$ be such that $\varrho$ is even, $\varrho(t)=1$ for $|t|\le t_{2}$, $\varrho(t)=0$ for
$|t| \ge 2t_{2}$, $|\varrho^{\prime}(t)|\le 2/t_{2}$, and
$\varrho^{\prime}(t)t\le 0.$ We then define the modified function $\widetilde{f}: \Omega\times\mathbb{R}\to \mathbb{R}$ as
\begin{equation*}
{\widetilde f}(x,t):=\frac{\partial}{\partial t}{\widetilde
	F}(x,t),
\end{equation*}
where
$${\widetilde F}(x,t):=\varrho(t)F(x,t)+(1-\varrho(t))\beta |t|^{p^{-}}$$
for some fixed $\beta\in \round{0,\min\left\{\frac{1}{p^-},\frac{1}{p^+2^{p^-}C_{imb}^{p^-}}\right\}}$ with $C_{imb}$ being the imbedding constant for the imbedding $W_0^{s,p(\cdot,\cdot)}(\Omega)\hookrightarrow L^{p^-}(\Omega).$ Clearly, ${\widetilde F}$ is even in $t$,
\begin{equation}\label{f.tilde}
{\widetilde f}(x,t)=\varrho^{\prime}(t)F(x,t)+\varrho(t)f(x,t)-\varrho^{\prime}(t)\beta|t|^{p^-}+(1-\varrho(t))\beta p^-|t|^{p^--2}t,
\end{equation}
and
\begin{align*}
p^-{\widetilde F}(x,t)-{\widetilde f}(x,t)t=\varrho(t)\big[p^-F(x,t)-f(x,t)t\big]-\varrho'(t)t\big[F(x,t)-\beta|t|^{p^-}\big].
\end{align*}
Thus, the definition of $\varrho$ and \eqref {F1.near.0} yield
\begin{equation}\label{F.tilde.nonegative1}
p^-{\widetilde F}(x,t)-{\widetilde f}(x,t)t\geq 0\quad \text{for a.e.}\ x\in\Omega\ \ \text{and all}\ \ t\in\mathbb{R},
\end{equation}
and
\begin{equation}\label{F.tilde.nonegative2}
p^-{\widetilde F}(x,t)-{\widetilde f}(x,t)t=0\quad \text{if and only if}\ \ t=0 \ \ \text{or}\ \ |t|\geq 2t_2.
\end{equation}
We now consider the modified energy functional $\widetilde{E}:\ W_0^{s,p(\cdot,\cdot)}(\Omega)\to\mathbb{R}$ given by
$$\widetilde{E}(u):=\Phi(u)-\int_{\Omega}{{\widetilde F}(x,u)}\,\diff x,\quad u\in W_0^{s,p(\cdot,\cdot)}(\Omega),$$
where $\Phi$ is defined as in \eqref{Phi}. By the definition of $\varrho$ and ($\textup{F}1$), we deduce from \eqref{f.tilde} that there exists a positive constant $C$ such that
\begin{equation}\label{est.tilde}
{\widetilde F}(x,t)\leq C+\beta|t|^{p^-}\ \ \text{and}\quad |{\widetilde f}(x,t)|\leq C\round{1+|t|^{p^--1}}\ \text{for a.e.}\ x\in\Omega\ \text{and all}\ t\in\mathbb{R}.
\end{equation}
Then by a standard argument, invoking the imbedding $W_0^{s,p(\cdot,\cdot)}(\Omega)\hookrightarrow L^{p^-}(\Omega)$ and the differentiability of $\Phi$, we can show that $\widetilde{E}\in C^1\round{W_0^{s,p(\cdot,\cdot)}(\Omega),\mathbb{R}}$. Obviously, $\widetilde{E}$ is even on $W_0^{s,p(\cdot,\cdot)}(\Omega)$ and $\widetilde{E}(0)=0.$

Next, we shall show that $\widetilde{E}$ is coercive on $W_0^{s,p(\cdot,\cdot)}(\Omega).$
By Proposition~\ref{norm-modular2}, the relations \eqref{equivalent.norms} and \eqref{est.tilde}, we have
\begin{align*}
\widetilde{E}(u)&\geq \frac{1}{p^+}\round{|u|_{s,p}^{p^-}-1}-\beta\|u\|^{p^-}_{L^{p^-}(\Omega)}-C|\Omega|\\
&\geq \frac{1}{p^+2^{p^-}}\|u\|_{s,p}^{p^-}-\beta C_{imb}^{p^-}\|u\|_{s,p}^{p^-}-C|\Omega|-\frac{1}{p^+}.
\end{align*}
This infers the coerciveness and the boundedness from below of $\widetilde{E}$ on $W_0^{s,p(\cdot,\cdot)}(\Omega)$ since $\beta\in \round{0,\min\left\{\frac{1}{p^-},\frac{1}{p^+2^{p^-}C_{imb}^{p^-}}\right\}}.$ To see that $\widetilde{E}$ satisfies the $(\textup{PS})$ condition, we first note that the operator $u\mapsto \int_\Omega\widetilde{f}(x,u)\diff x$ is compact due to the subcritical growth condition \eqref{est.tilde} and the compactness of the imbedding  $W_0^{s,p(\cdot,\cdot)}(\Omega)\hookrightarrow L^{p^-}(\Omega)$.  From  this, the coerciveness of $\widetilde{E}$ and the $(S_+)$ property of $\Phi'$ (see Lemma~\ref{s+}), we easily deduce that $\widetilde{E}$ satisfies the $(\textup{PS})$ condition.

We now verify that $\widetilde{E}$ fulfills the last condition in Lemma~\ref{lemma_last}. Let $n\in\mathbb{N}$ be arbitrary and fixed.  Let $\phi_1,\cdots,\phi_n$ be  functions in $C_c^\infty(\Omega),$ that are linearly independent. Set $X_{n}:= {\rm
	span}\left\{\phi_{1},...,\phi_{n}\right\}$. Then, norms $\|\cdot\|_{L^\infty(\Omega)}, \|\cdot\|_{L^{p^-}(\Omega)}$ and $\|\cdot\|_{s,p}$ are equivalent on $X_n$ since $X_n$ is finitely dimensional. That is, there exist positive constants $C_{n,1}$ and $C_{n,2}$ such that
\begin{equation}\label{equi,norms}
C_{n,1}\|u\|_{L^\infty(\Omega)}\leq \|u\|_{s,p} \leq C_{n,2} \|u\|_{L^{p^-}(\Omega)},\quad \forall u\in X_n.
\end{equation}
By ($\textup{F}2$) and ($\textup{F}3$), there exists $t_3\in (0,t_2)$ such that
\begin{equation}\label{F1.last}
F(x,t)\geq \frac{2^{p^-+1}C_{n,2}^{p^-}}{p^-}|t|^{p^-},\quad \text{for a.e.}\ \ x\in\Omega\ \ \text{and all}\ \ |t|\leq t_3.
\end{equation}
Choose $r_n:=\min\left\{\frac{1}{2},t_3C_{n,1}\right\}.$ Then, for $u\in X_n$ with $\|u\|_{s,p}=r_n$ we have $|u|_{s,p}\leq 1$ and $\|u\|_{L^\infty(\Omega)}\leq t_3$ due to the relations \eqref{equivalent.norms} and \eqref{equi,norms} and hence, invoking Proposition~\ref{norm-modular2} and \eqref{F1.last} with noting $\widetilde{F}(x,u)=F(x,u)$ for $\|u\|_{L^\infty(\Omega)}\leq t_3$, we derive
\begin{align*}
\widetilde{E}(u)&\leq \frac{1}{p^-}|u|_{s,p}^{p^-}-\frac{2^{p^-+1}C_{n,2}^{p^-}}{p^-}\|u\|_{L^{p^-}(\Omega)}^{p^-}\\
&\leq -\frac{2^{p^-}}{p^-}\|u\|_{s,p}^{p^-}=-\frac{(2r_n)^{p^-}}{p^-},\quad \forall u\in X_n\cap S_{r_n}.
\end{align*}
We therefore obtain
	$$ \sup_{u\in X_{n}\cap S_{r_{n}}}{\widetilde{E}(u)}<0.$$
Thus, invoking Lemma \ref{lemma_last} we deduce a sequence $\{u_n\}\subset W_0^{s,p(\cdot,\cdot)}(\Omega)$ such that $\widetilde{E}'(u_n)=0$ for all $n\in\mathbb{N}$ and  $\widetilde{E}(u_n)\to 0$ as $n\to\infty.$ Since $\widetilde{E}$ satisfies the $(\textup{PS})$ condition, we can extract from $\{u_n\}$ a subsequence, still denote by $\{u_n\}$, such that $u_n\to \bar{u}$ in $W_0^{s,p(\cdot,\cdot)}(\Omega).$ Since $\widetilde{E}\in C^1\round{W_0^{s,p(\cdot,\cdot)}(\Omega),\mathbb{R}},$ we obtain
$$\widetilde{E}(\bar{u})=\langle \widetilde{E}'(\bar{u}),\bar{u}\rangle =0.$$
This and \eqref{F.tilde.nonegative1} yield
\begin{align*}
0\leq\int_{\mathbb{R}^N}\int_{\mathbb{R}^N}\left(\frac{1}{p^-}-\frac{1}{p(x,y)}\right)
\frac{|\bar{u}(x)-\bar{u}(y)|^{p(x,y)}}{|x-y|^{N+sp(x,y)}}
\,\diff x \diff y+\int_\Omega\left(\frac{1}{p^-}-\frac{1}{p(x)}\right)|\bar{u}|^{p(x)}\diff x\\
=\int_\Omega \left[{\widetilde F}(x,\bar{u}(x))-\frac{1}{p^-}{\widetilde f}(x,\bar{u}(x))\bar{u}(x)\right]\diff x\leq 0.
\end{align*}
From this and \eqref{F.tilde.nonegative2} we have
$$\int_{\mathbb{R}^N}\int_{\mathbb{R}^N}\left(\frac{1}{p^-}-\frac{1}{p(x,y)}\right)
\frac{|\bar{u}(x)-\bar{u}(y)|^{p(x,y)}}{|x-y|^{N+sp(x,y)}}
\,\diff x \diff y+\int_\Omega\left(\frac{1}{p^-}-\frac{1}{p(x)}\right)|\bar{u}|^{p(x)}\diff x=0,$$
and for a.e. $x\in\Omega,$ $\bar{u}(x)=0$ or $|\bar{u}(x)|\geq 2t_2.$ Thus, ${\widetilde F}(x,\bar{u}(x))=0$ or $\beta |\bar{u}(x)|^{p^-}$ and furthermore, for a.e. $x$\ \  satisfying $|\bar{u}(x)|\geq 2t_2$, we have $p(x)=p^-.$ From these facts, we have

\begin{align*}
0= \widetilde{E}(\bar{u})&\geq \int_\Omega\frac{1}{p(x)}|\bar{u}|^{p(x)}\diff x-\int_\Omega {\widetilde F}(x,\bar{u}(x))\diff x=\int_\Omega\frac{1}{p^-}|\bar{u}|^{p^-}\diff x-\int_\Omega {\widetilde F}(x,\bar{u}(x))\diff x\\
&\geq \int_\Omega\frac{1}{p^-}|\bar{u}|^{p^-}\diff x-\int_\Omega\beta|\bar{u}|^{p^-}\diff x,
\end{align*}
and hence, $\bar{u}=0$ since $\beta\in \round{0,\min\left\{\frac{1}{p^-},\frac{1}{p^+2^{p^-}C_{imb}^{p^-}}\right\}}.$ That is, we have derived $u_n\to 0$ in $W_0^{s,p(\cdot,\cdot)}(\Omega),$ and hence $\|u_n\|_{L^{p(\cdot)}(\Omega)}\to 0$ as $n\to\infty.$ Note that $\{u_n\}$ are weak solutions to the following probem
\begin{equation*}
\begin{cases}
(-\Delta)_{p(x)}^su=f_1(x,u) \quad \text{in} \quad \Omega, \\
u=0 \quad \text{in} \quad \Bbb R^N\setminus \Omega,
\end{cases}
\end{equation*}
where the nonlinear term $f_1(x,t):=\widetilde{f}(x,t)-|t|^{p(x)-2}u$ fulfills the condition $(\textup{F}1)$ with $q(x)=p(x)$ due to the relation in \eqref{est.tilde}.  It yields $\norm{u_{n}}_{L^{\infty}(\Omega)}\to 0$ in view of Theorem \ref{a-priori.bound}. Thus, $\norm{u_{n}}_{L^{\infty}(\Omega)}\le t_2$ for large $n$ amd hence, $\left\{u_{n}\right\}$ with $n$ large enough are weak solutions of problem \eqref{eq.existence}. The proof is complete.
\end{proof}

\section*{Acknowledgements}
The second author was supported by the Basic Science Research Program through the National Research Foundation of Korea (NRF) funded by the Ministry of Education (NRF-2016R1D1A1B03935866).

\end{document}